\documentclass{article}

\usepackage{arxiv}

\usepackage[utf8]{inputenc} % allow utf-8 input
\usepackage[T1]{fontenc}    % use 8-bit T1 fonts
\usepackage{hyperref}       % hyperlinks
\usepackage{url}            % simple URL typesetting
\usepackage{amsfonts}       % blackboard math symbols
\usepackage{nicefrac}       % compact symbols for 1/2, etc.
\usepackage{microtype}      % microtypography
\usepackage{lipsum}

\usepackage{array} % For advanced table formatting (m-column, centering)
\usepackage{booktabs} % For improved table aesthetics (optional) 
\usepackage{caption} % For controlling table captions (optional)

\usepackage{graphicx}
\graphicspath{ {./images/} }
\usepackage[numbers]{natbib}
\usepackage{hyperref}
\usepackage{url}
\usepackage[utf8]{inputenc}
\usepackage{amssymb}
\usepackage{amsmath}
\usepackage{amsthm}
\usepackage{mathtools}

\usepackage{subcaption}
\usepackage{hyperref}
\usepackage[dvipsnames]{xcolor}
\usepackage{enumitem}
\usepackage{subcaption}
\usepackage{comment}

\theoremstyle{definition}
\newtheorem{theorem}{Theorem}[section]

\newtheorem{lemma}[theorem]{Lemma}
\newtheorem{assumption}[theorem]{Assumption}
\newtheorem{remark}[theorem]{Remark}
\newtheorem{definition}[theorem]{Definition}

\usepackage[colorinlistoftodos]{todonotes}
\usepackage{afterpage}
\usepackage{bm}
\usepackage{algorithm}
\usepackage{algpseudocode}
\usepackage{float}

\title{A Data-Driven Framework for Koopman Semigroup Estimation in Stochastic Dynamical Systems}
\date{}

\author{
Yuanchao Xu\thanks{Corresponding author. Department of Mathematical and Statistical Sciences, University of Alberta, University Commons 5-140
Edmonton, Alberta, Canada T6G 2N8. (yuanchao@ualberta.ca)} \quad
Kaidi Shao\thanks{International Center for Primate Brain Research, Chinese Academy of Sciences, No. 500 Qiang Ye Road, Songjiang District, Shanghai, 201602, China. (kaidi.shao@icpbr.ac.cn)} \quad
Isao Ishikawa\thanks{Center for Science Adventure and Collaborative Research Advancement, Kyoto University, Kitashirakawa Oiwake-cho, Sakyo-ku, Kyoto-shi, Kyoto 606-8502, Japan. (ishikawa.isao.5s@kyoto-u.ac.jp)} \quad
Yuka Hashimoto\thanks{NTT Network Service Systems Laboratories, NTT Corporation, 3-9-11, Midori-cho, Musashinoshi, Tokyo, 180-8585, Japan. (yuka.hashimoto@ntt.com)} \quad
Nikos Logothetis\thanks{International Center for Primate Brain Research, Chinese Academy of Sciences, No. 500 Qiang Ye Road, Songjiang District, Shanghai, 201602, China. (nikos.logothetis@icpbr.ac.cn)} \quad
Zhongwei Shen\thanks{Department of Mathematical and Statistical Sciences, University of Alberta, University Commons 5-140
Edmonton, Alberta, Canada T6G 2N8. (zhongwei@ualberta.ca)}
}

\begin{document}
\maketitle

\begin{abstract}
    We present Stochastic Dynamic Mode Decomposition (SDMD), a novel data-driven framework for approximating the Koopman semigroup in stochastic dynamical systems. Unlike existing approaches, SDMD explicitly incorporates sampling time into its formulation to ensure numerical stability and precision in the presence of noise. By directly approximating the Koopman semigroup rather than its generator, SDMD avoids computationally expensive matrix exponential calculation, providing a more practically efficient pathway for analyzing stochastic dynamics. The framework also leverages neural networks for automated basis selection, minimizing manual effort while preserving computational efficiency. We establish SDMD's theoretical foundations through rigorous convergence guarantees across three critical limits in order: large data, infinitesimal sampling time, and increasing dictionary size. Numerical experiments on canonical stochastic systems including oscillatory system, mean-reverting processes, metastable system and a neural mass model demonstrate SDMD's effectiveness in capturing the spectral properties of the Koopman semigroup, even in systems with complex random behavior.
\end{abstract}

\keywords{stochastic Koopman operator, Markov semigroup, data-driven dynamical system, dynamic mode decomposition, perturbation theory, machine learning}

%%%% Section 1 %%%%
\section{Introduction}

Dynamical systems theory \cite{strogatz2024nonlinear} provides a foundational framework for modeling complex phenomena across different areas such as climate science \cite{2021AGUFM.A15E1677G}, molecular dynamics \cite{schutte2023overcoming}, fluid mechanics \cite{anderson1995computational}, finance \cite{Mann01112016}, and neuroscience \cite{izhikevich2007dynamical}. In these fields, stochastic dynamical systems naturally arise due to the important role of random perturbation. Consideration of stochasticity is essential because real-world systems often feature incomplete data, measurement noise, high degrees of freedom, and complex, unknown mechanisms. Random fluctuation can drive critical transitions like abrupt climate shifts \cite{benzi1982stochastic} or molecular conformational changes \cite{bolhuis2002transition} that deterministic models often fail to capture. This is also evident in neuroscience, where stochastic modeling methods are increasingly preferred over deterministic approach \cite{faisal2008noise, tuckell1988introduction} due to intrinsic variability in neural activity from synaptic release \cite{manwani1999detecting}, ion-channel gating \cite{hawkes2004stochastic}, and fluctuating inputs \cite{destexhe1999impact}. Moreover, neural dynamics are inherently multi‑scale, ranging from millisecond ion‑channel fluctuations to slow population‑level oscillations, and stochastic approaches remain one of the few tractable ways to bridge these scales \cite{deco2008dynamic}. In general, models accounting for stochasticity usually offer a more realistic and mathematically robust representation of brain dynamics.

Practically speaking, the lack of complete knowledge about underlying first-principles makes direct modeling of complex systems extremely challenging. However, data-driven methods have gained significant attention in recent years for their ability to directly extract insights from observations without requiring detailed prior knowledge or explicit mathematical formulations \cite{brunton2016discovering,chen2018neural,li2020fourier,lu2019deeponet,rudy2017data}. This approach is particularly valuable for complex stochastic systems where theoretical models are incomplete or computationally infeasible. 

Among various data-driven approaches, operator-theoretic methods excel as powerful tools for dynamical system analysis. In particular, the Koopman operator theory \cite{Koopman1931, koopman1932dynamical} converts nonlinear dynamics into a linear, though potentially infinite-dimensional, framework through observables. This conversion allows spectral analysis \cite{mezic2005spectral, mezic2013analysis, rowley2009spectral}, where eigenvalues and eigenfunctions reveal crucial details about a system's stability and long-term behavior. In stochastic systems, the operator becomes a Markov semigroup defined by conditional expectations and captures probabilistic evolution over time. Eigenvalues indicate decay rates and growth patterns, while eigenfunctions help identify invariant structures and coherent features within the dynamics. Such spectral insights are invaluable for understanding system behavior, predicting future evolution, and designing control strategies in both deterministic and stochastic settings. Recent developments in data-driven methods for estimating Koopman semigroup or generator \cite{colbrook2024rigorous,ishikawa2024koopman,kostic2022learningdynamicalsystemskoopman,noe2013variational,Tu2014,Williams_2015,xu2025reskoopnetlearningkoopmanrepresentations}, particularly Extended Dynamic Mode Decomposition (EDMD) \cite{Williams_2015}, have made significant progress in approximating the Koopman operator directly from data. However, EDMD was originally designed for deterministic systems and does not explicitly account for stochastic effects. To address these limitations, various methods \cite{Colbrook2023BeyondER,vcrnjaric2020koopman,KLUS2020132416,kostic2022learningdynamicalsystemskoopman,noe2013variational,wanner2022robust} have been proposed over time. For example, in \cite{Colbrook2023BeyondER}, authors have introduced the concept of variance-pseudospectra as a measure of statistical coherency, which helps in understanding the stochastic system's spectral properties. In \cite{vcrnjaric2020koopman}, authors introduced stochastic Hankel-DMD (sHankel-DMD) algorithm to approximate the spectral properties of the stochastic Koopman operator. In \cite{KLUS2020132416} authors generalized Galerkin approximation method as an extension of EDMD (gEDMD) to approximate the infinitesimal generator of the Koopman operator. In \cite{kostic2022learningdynamicalsystemskoopman}, authors developed a statistical learning framework to learn Koopman operators in reproducing kernel Hilbert spaces (RKHS). In \cite{noe2013variational}, authors proposed a variational approach based on maximizing the Rayleigh coefficient for modeling slow processes in stochastic dynamical systems. In \cite{wanner2022robust}, authors introduced a new DMD algorithm that can accurately approximate the stochastic Koopman operator even when both the dynamics are random and the measurements contain noise. It also enables time-delayed observables for random systems using data from a single trajectory.

In this paper, we introduce \textbf{Stochastic Dynamic Mode Decomposition (SDMD)}, a novel data-driven framework that estimates the Koopman semigroup in the stochastic system by explicitly incorporating sampling time into the approximation process. The key innovation of directly approximating the Koopman semigroup bypass the need for matrix exponential computations. This design not only enhances computational efficiency but also ensures numerical stability when dealing with a typically unbounded Koopman generator. The main contributions of this work include:
\begin{itemize}
    \item \textbf{Explicit Consideration of Sampling Time (\(\Delta t\)) for Stability}: The explicit inclusion of sampling time (\(\Delta t\)) in the SDMD framework is a key innovation, which ensures numerical stability and precision, and addresses challenges faced by other methods in handling stochastic dynamics.

    \item \textbf{Direct Approximation of the Semigroup}: SDMD directly approximates the Koopman semigroup, which avoids the computationally expensive matrix exponential calculations required by most methods that return only the generator. This approach reduces computational cost while providing a more practical and efficient pathway for analyzing stochastic systems.

    \item \textbf{Computational Efficiency with Neural Network Integration}: The neural network extension enables adaptive basis selection directly from data without requiring manual intervention. Unlike other methods which may involves resource-intensive computations of Jacobian and Hessian matrices, our method significantly reduces computational resource while maintaining consistency with the stochastic evolution.
    
    \item \textbf{Rigorous Theoretical Guarantees}: The proposed framework includes comprehensive convergence analysis, covering the large data limit, the zero-limit of sampling time, and the large dictionary size. These guarantees establish the reliability and robustness of SDMD in approximating the Koopman semigroup.
\end{itemize}

From an application perspective, our SDMD framework aims to provide a practical and efficient tool for analyzing stochastic dynamical systems. By explicitly incorporating sampling time and directly approximating the Koopman semigroup, SDMD provides more reliable spectral analysis for systems where random fluctuations play a crucial role. This enhanced reliability translates to improved identification of dominant modes, more accurate prediction of system evolution, and better characterization of system stability under stochastic perturbations. The method's computational efficiency makes it particularly suitable for analyzing large datasets from experimental observations, where both noise and deterministic dynamics need to be properly accounted for. Additionally, the neural network extension of SDMD offers automated feature extraction from complex data, reducing the need for domain-specific expertise in basis function selection. These practical benefits complement the theoretical guarantees of our approach, making SDMD a powerful tool for researchers seeking to understand the fundamental dynamical properties of stochastic systems across various scientific disciplines.

The rest of this paper is organized as follows: Section 2 provides the mathematical background of stochastic Koopman operators. Section 3 details our computational methodology. Section 4 presents the convergence analysis. Section 5 extends the framework to neural network implementations. Section 6 demonstrates the effectiveness of our approach through experiments. Finally, Section 7 concludes with discussions and future directions.

%%%% Section 2 %%%%
\section{Stochastic Koopman Operator}\label{sec: intro_stochastic_kpm_operator}

In dynamical systems, the Koopman operator provides a powerful mathematical framework for analyzing the evolution of observables instead of the system states themselves. For stochastic systems, the Koopman operator forms a Markov semigroup \cite{lasota2013chaos,pavliotis2016stochastic,pazy2012semigroups} defined through conditional expectations, capturing the probabilistic evolution of observables over time by describing how their expected values change under the influence of both deterministic dynamics and random perturbations.

Let $\mathcal{M} \subseteq \mathbb{R}^d$ be the state space equipped with the Borel $\sigma$-algebra, and consider a continuous-time stochastic process $(\mathbf{X}_t)_{t\geq 0}$ on a probability space $(\Omega,\mathbb{P})$ defined by the stochastic differential equation:
\begin{equation}\label{eq: sde}
    d\mathbf{X}_t = \mathbf{b}(\mathbf{X}_t)dt + \bm{\sigma}(\mathbf{X}_t)d\mathbf{W}_t, \quad \mathbf{X}_0 = \mathbf{x},
\end{equation}
where $\mathbf{b}: \mathcal{M} \to\mathbb{R}^d$ is the drift term, $\bm{\sigma}: \mathcal{M} \to\mathbb{R}^{d\times d}$ is the diffusion term, and $(\mathbf{W}_t)_{t\geq 0}$ is an $d$-dimensional Wiener process. We assume that both $\mathbf{b}$ and $\sigma$ satisfy appropriate regularity condition \cite{engel1999one}.

Let $\rho$ be a probability distribution on $\mathcal{M}$. The space $\mathcal{F}$ of square-integrable functions with respect to $\rho$ is defined as:
\begin{equation*}
    \mathcal{F} \coloneqq \left\{f : \int_\mathcal{M} |f|^2 d\rho < \infty\right\},
\end{equation*}
equipped with the inner product 
$\langle f,g \rangle_\rho \coloneqq \int_\mathcal{M} f g \; d\rho$ and corresponding norm $\|f\|_{\rho} \coloneqq \sqrt{\langle f,f \rangle_\rho}$. Notice that $\rho$ is not necessarily a stationary distribution of the underlying dynamical system. 

For any observable $f \in \mathcal{F}$, the stochastic Koopman operator family $(\mathcal{K}^t)_{t\geq0}$ is defined as
\begin{equation}\label{def:stochastic_kpm}
    (\mathcal{K}^t f)(\mathbf{x}) \coloneqq \mathbb{E}_{\mathbb{P}}[f(\mathbf{X}_t)|\mathbf{X}_0=\mathbf{x}],
\end{equation}
where $\mathbb{E}_{\mathbb{P}}$ denotes the expectation with respect to the probability measure $\mathbb{P}$ on $\Omega$, and $\mathbf{X}_t:\Omega\to \mathcal{M}$ is the process starting from $\mathbf{x}$.

\begin{assumption}\label{ass:str_cont_semigroup_K}
We assume that $\{\mathcal{K}^t\}_{t \geq 0}$ is a strongly continuous semigroup of bounded linear operators on $\mathcal{F}$, that is, 
\begin{itemize}
    \item $\mathcal{K}^t$ is a bounded linear operator on $\mathcal{F}$ for each $t\geq0$;
    \item $\mathcal{K}^0 = \text{I}$, $\mathcal{K}^t \circ \mathcal{K}^s = \mathcal{K}^{t+s}$ for all $ t,s \geq 0$;
    \item $\lim_{t\to 0^{+}}\|\mathcal{K}^tf - f\|_{\rho} = 0$ for each $f \in \mathcal{F}$.
\end{itemize}
\end{assumption}

The connection between the stochastic process and Koopman operator can be further understood through its infinitesimal generator $\mathcal{A}$, defined as
\begin{equation}\label{def:generator}
    \mathcal{A}f \coloneqq \lim_{t\to 0}\frac{\mathcal{K}^t f-f}{t}
\end{equation}
on the domain $\mathcal{D}(\mathcal{A}) = \left\{ f \in \mathcal{F} : \lim_{t \to 0} \frac{\mathcal{K}^t f - f}{t} \text{ exists in } \mathcal{F} \right\}$.
\begin{remark}
    Assumption \ref{ass:str_cont_semigroup_K} ensures that the domain $ \mathcal{D}(\mathcal{A})$ is dense in $\mathcal{F}$ and the generator \(\mathcal{A}\) is a closed operator \cite{pazy2012semigroups}. This assumption is typically satisfied by a broad class of stochastic dynamical systems when the drift $\mathbf{b(\cdot)}$ and diffusion $\bm{\sigma(\cdot)}$ coefficients satisfy certain regularity conditions and solutions are well-posed, e.g., a gradient potential system where the potential grows sufficiently fast at infinity. We will rely on Assumption \ref{ass:str_cont_semigroup_K} throughout this paper, especially in Section \ref{cvg_pj_kpm_generator}. 
\end{remark}

From It\^{o}'s formula \cite{pavliotis2016stochastic}, we have:
\begin{equation}\label{eq: generator_ito_formula}
    \mathcal{A}f = \sum_{i=1}^d \mathbf{b}_i\frac{\partial f}{\partial x_i} + \frac{1}{2}\sum_{i,j=1}^d (\bm{\sigma}\bm{\sigma}^T)_{ij}\frac{\partial^2 f}{\partial x_i\partial x_j},\quad \forall f\in C^2_b(\mathcal{M}),
\end{equation}
where $C^2_b(\mathcal{M})$ denotes the space of twice continuously differentiable functions with bounded derivatives on $\mathcal{M}$.

For spectral analysis of the Koopman generator $\mathcal{A}$, we consider the eigenvalue problem:
\begin{equation*}
    \mathcal{A}\phi = \lambda\phi,
\end{equation*}
where $\lambda\in\mathbb{C}$ and $\phi\in\mathcal{D}(\mathcal{A})$ are the eigenvalue and eigenfunction respectively. The eigenvalue $\lambda$ of the Koopman generator $\mathcal{A}$ is closely connected to the eigenvalue $\mu$ of the stochastic Koopman operator $\mathcal{K}^t$ through the following relationship:
\begin{equation}\label{eq: evalue_generator_semigroup}
    \mu = e^{t\lambda}.
\end{equation}
This relationship provides a practical way to compute the generator's spectrum from discrete-time observations \cite{engel1999one}.

%%%% Section 3 %%%%
\section{Computation Method in Stochastic Dynamical System}

This section presents a computational method for analyzing stochastic dynamical systems through the lens of Koopman operator theory. The core idea involves utilizing the stochastic Taylor expansion \cite[section 5.2]{pavliotis2016stochastic} within the framework of the Galerkin method \cite{boyd2013chebyshev}, which results in an approach tailored for stochastic systems. This approach, referred to as Stochastic Dynamic Mode Decomposition (SDMD), provides a data-driven framework for approximating the Koopman semigroup of stochastic systems. Below, we introduce some relevant background and the necessary notation.

\textbf{Notation: }Let \( \{ \psi_1, \dots, \psi_N \}\subset\mathcal{D}(\mathcal{A}) \) be a set of dictionary functions defined on the state space \( \mathcal{M} \), forming the finite-dimensional space \( \mathcal{F}_N \coloneqq \text{span} \{\psi_1, \dots, \psi_N \} \). For these functions, we define the following Gram matrices $G, H \in \mathbb{R}^{N\times N}$
\begin{equation}\label{gram}
    [G]_{ij} \coloneqq \langle\psi_i,\psi_j\rangle_\rho, \quad [H]_{ij} \coloneqq \langle\psi_i,\mathcal{A}\psi_j\rangle_\rho,
\end{equation}
where the dictionary functions are assumed to be linearly independent in order to satisfy $G$'s invertibility.

In practice, let $\{\mathbf{x}_k\}_{k=1}^m$ be the i.i.d. data sampled from $\rho$, i.e., each $\mathbf{x}_i$ is drawn independently and identically from the probability distribution $\rho$.
Next, construct the data matrices $\Psi_X, \Psi^{'}_X\in\mathbb{R}^{m\times N}$ in the following:
\begin{gather}\label{data_matrices}
    \Psi_X \coloneqq 
    \begin{bmatrix}
        \psi_1(\mathbf{x}_1) & \cdots & \psi_{N}(\mathbf{x}_1) \\
        \vdots & \ddots & \vdots \\
        \psi_1(\mathbf{x}_m) & \cdots & \psi_{N}(\mathbf{x}_m)
    \end{bmatrix} , \
    \Psi^{'}_X \coloneqq 
    \begin{bmatrix}
        \mathcal{A}\psi_1(\mathbf{x}_1) & \cdots & \mathcal{A}\psi_{N}(\mathbf{x}_1) \\
        \vdots & \ddots & \vdots \\
        \mathcal{A}\psi_1(\mathbf{x}_m) & \cdots & \mathcal{A}\psi_{N}(\mathbf{x}_m)
    \end{bmatrix}.
\end{gather}
\begin{remark}
    Since we manually pick up the basis functions, we can directly obtain the Jacobian and Hessian matrices as required for computing each $\mathcal{A}\psi_j(\mathbf{x}_i)$. However, in Neural Network based method,  basis functions can be trained from a time series by Automatic Differentiation \cite{JMLR:v18:17-468}. More details of such method will be discussed in Section \ref{sec:sdmd_dl}.
\end{remark}

Thus, the Gram matrices $G$ and $H$ can be estimated empirically from these data matrices. Specifically, we construct
\begin{equation}\label{eq: empirical_gram}
    \widehat{G} = \frac{1}{m}\Psi_X^*\Psi_X, \quad \widehat{H} = \frac{1}{m}\Psi_X^*\Psi^{'}_X.
\end{equation}
where $*$ denotes conjugate transpose of a matrix. These empirical approximation converges to their theoretical counterparts as the amount of data increases as discussed in \cite{KLUS2020132416,Williams_2015}. Specifically, by the Strong Law of Large Numbers (SLLN), we have
\begin{equation}\label{eq: empirical_gram_lim}
    \lim_{m\to\infty} [\widehat{G}]_{ij} = [G]_{ij}, \quad \lim_{m\to\infty} [\widehat{H}]_{ij} = [H]_{ij} \quad \text{a.s.}
\end{equation}
Building on these definitions, the following section introduces the SDMD method, which combines the Galerkin approximation framework with stochastic dynamics.
\begin{remark}
    The construction of the matrices $\widehat{G}$ and $\widehat{H}$ in Eq.~\eqref{eq: empirical_gram} follows essentially the same procedure as in the gEDMD framework, where $\widehat{G}$ approximates the Gram matrix of the chosen dictionary functions and $\widehat{H}$ approximates the action of the generator on this dictionary. The main distinction lies in that SDMD estimates the semigroup instead of generator, but the underlying matrix assembly is consistent with gEDMD. See \cite{KLUS2020132416} for more details.
\end{remark}

%%%% Section 3.1 %%%%

\subsection{Stochastic Dynamic Mode Decomposition (SDMD)}\label{sdmd}
The SDMD method approximates the stochastic Koopman operator by incorporating the stochastic Taylor expansion \cite{pavliotis2016stochastic} into the EDMD framework. The method explicitly includes the sampling time $\Delta t$ and truncates higher-order terms in the expansion. Specifically, SDMD estimates the Koopman operator as follows:
\begin{equation}\label{eq: sdmd_kpm_operator_2}
    \widehat{K}_{N,\Delta t,m} \coloneqq I + \Delta t \  \widehat{G}^{-1}\widehat{H},
\end{equation}
where $\widehat{G}$ and $\widehat{H}$ are the Gram matrices computed from data as in Eq.\eqref{eq: empirical_gram}. This formulation omits higher-order terms for sufficiently small $\Delta t$, enabling an efficient approximation of the stochastic Koopman operator. Below, we provide the derivation that leads to this result.

\subsubsection*{Derivation of SDMD Method}
Consider the stochastic system defined in Eq.\eqref{eq: sde}. Let 
$$
    \bm{\Psi}_N(\mathbf{x}_i)=[\psi_1(\mathbf{x}_i) \ \dots \ \psi_N(\mathbf{x}_i)]^\top
$$
be a vector of manually selected basis functions evaluated at some data point $\mathbf{x}_i$. Suppose $f(\mathbf{x}_i)=\bm{\Psi}_N(\mathbf{x}_i)^\top \bm{a}$ for some $\bm{a} \in \mathbb{R}^N$, then for any $\Delta t>0$, EDMD \cite{Williams_2015} approximates the Koopman operator $\mathcal{K}^{\Delta t}$ onto $\mathcal{F}_N$ using Galerkin approximation:
\begin{align}\label{expansion_edmd}
    \mathcal{K}^{\Delta t} f(\mathbf{x}_i)
    &= \bm{\Psi}_N(\mathbf{x}_i)^\top \widehat{K}_{N, m}\bm{a} + r(\mathbf{x}_i),
\end{align}
where $r(\mathbf{x}_i)$ is the residual.

While EDMD provides a framework for approximating the Koopman operator, our SDMD method explicitly addresses the challenges of stochastic systems by incorporating stochastic Taylor expansion \cite{pavliotis2016stochastic} in Eq.\eqref{def:generator} to account for noise. Specifically, for each basis function $\psi_j$ evaluated at $\mathbf{x}_i$
\begin{equation*}
    \mathcal{A}\psi_j(\mathbf{x}_i) 
    \approx \frac{\mathcal{K}^{\Delta t} \psi_j(\mathbf{x}_i)-\psi_j(\mathbf{x}_i)}{\Delta t}.
\end{equation*}
After rearrange, we have
\begin{equation}\label{eq: expansion_basis}
    \mathcal{K}^{\Delta t} \psi_j(\mathbf{x}_i)
    \approx \psi_j(\mathbf{x}_i) + \Delta t \ \mathcal{A}\psi_j(\mathbf{x}_i) + o_{i,j}(\Delta t),
\end{equation}
where each $\mathcal{A}\psi_j(\mathbf{x}_i)$ is computed by Eq.\eqref{eq: generator_ito_formula} and $o_{i,j}(\Delta t)$ is the asymptotic term $o(\Delta t)$ corresponding to $\mathcal{K}^{\Delta t}\psi_j$ expanded at data point $\mathbf{x}_i$. 

While SDMD truncates the stochastic Taylor expansion at the first-order term $\Delta t\,\mathcal{A}\psi_j$, the expansion can in principle be extended to include the second-order term $\frac{\Delta t^2}{2}\,\mathcal{A}^2\psi_j$. This would compute a higher-order approximation of the Koopman semigroup that potentially improves accuracy for larger sampling intervals $\Delta t$. However, computing $\mathcal{A}^2\psi_j$ requires higher-order derivatives of both the basis functions and the SDE coefficients, which greatly increases computational cost. The explicit form of $\mathcal{A}^2\psi_j$ for general SDEs is provided in Appendix~\ref{2nd_order_expansion}.

\begin{remark}
The asymptotic term $o(\Delta t)$ represents the remainder terms in the Taylor expansion that decay faster than $\Delta t$ as $\Delta t \to 0$. Specifically, for each basis function $\psi_j$ evaluated at data point $\mathbf{x}_i$, we have:
\begin{equation*}
    \lim_{\Delta t \to 0} \frac{o_{i,j}(\Delta t)}{\Delta t} = 0.
\end{equation*}
This notation is used to indicate that these terms become negligible compared to the linear term $\Delta t$ for sufficiently small sampling time steps, which justifies their omission in the approximation Eq.\eqref{eq: sdmd_kpm_operator_2} when $\Delta t$ is small.
\end{remark}

Next, we use Eq.\eqref{eq: expansion_basis} for each basis function $\psi_j$ in the expansion of the stochastic Koopman operator to approximate the expected value of $f$ in the stochastic dynamical system starting at $\mathbf{x}_i$, which gives:
\begin{align}\label{eq: expansion_sdmd}
    \mathcal{K}^{\Delta t}f(\mathbf{x}_i) 
    &= \sum_{j=1}^n \left[ \psi_j(\mathbf{x}_i) + \Delta t \ \mathcal{A}\psi_j(\mathbf{x}_i) + o_{i,j}(\Delta t) \right]\bm{a}_j \notag \\
    &= \left[ \bm{\Psi}_N(\mathbf{x}_i)^\top + \Delta t \ \mathcal{A}^{\bm{\Psi}}_N(\mathbf{x}_i)^\top + \bm{o}_{i,N}(\Delta t)^\top \right]\bm{a},
\end{align}
where $\bm{o}_{i,N}(\Delta t)=\left[o_{i,1}(\Delta t) \ \dots \ o_{i,N}(\Delta t)\right]^\top$ and $\mathcal{A}^{\bm{\Psi}}_N(\mathbf{x}_i) = \left[ \mathcal{A}\psi_1(\mathbf{x}_i) \ \dots \ \mathcal{A}\psi_N(\mathbf{x}_i) \right]^\top$.

Now, equating both Eq.\eqref{expansion_edmd} and Eq.\eqref{eq: expansion_sdmd}, then evaluating over all data points $\{\mathbf{x}_i\}_{i=1}^m$, we can have the following minimization problem:
\begin{equation*}
    \min_{\tilde{K}_{N,\Delta t,m}\in\mathbb{R}^{N\times N}} \sum_{i=1}^m \left| r(\mathbf{x}_i) \right|^2 
    = \min_{\tilde{K}_{N,\Delta t,m}\in\mathbb{R}^{N\times N}} \sum_{i=1}^m \left| \left[\bm{\Psi}_N(\mathbf{x}_i)^\top + \Delta t \ \mathcal{A}^{\bm{\Psi}}_N(\mathbf{x}_i)^\top + \bm{o}_{i,N}(\Delta t)^\top - \bm{\Psi}_N(\mathbf{x}_i)^\top \tilde{K}_{N,\Delta t,m}\right]\bm{a} \right|^2,
\end{equation*}
which is equivalent to: 
\begin{equation}\label{eq: min_loss}
    \min_{\tilde{K}_{N,\Delta t,m}\in\mathbb{R}^{N\times N}} \| \Psi_X + \Delta t \ \Psi^{'}_X +\bm{o}_{m,N}(\Delta t) - \Psi_X \tilde{K}_{N,\Delta t,m} \|_F^2,
\end{equation}
where $\| \cdot \|_F$ denotes the matrix Frobenius norm and the higher order matrix $\bm{o}_{m,N}(\Delta t)$ is $\left[\bm{o}_{m,N}(\Delta t)\right]_{ij} = o_{i,j}(\Delta t)$. More specifically,  
\begin{equation*}
    \bm{o}_{m,N}(\Delta t) = 
    \begin{bmatrix}
        o_{1,1}(\Delta t) & \cdots & o_{1,N}(\Delta t) \\
        \vdots & \ddots & \vdots \\
        o_{m,1}(\Delta t) & \cdots & o_{m,N}(\Delta t)
    \end{bmatrix}.
\end{equation*} 
Thus, the minimal $\tilde{K}_{N,\Delta t,m}$ is
\begin{align}\label{eq: sdmd_kpm_operator}
    \tilde{K}_{N,\Delta t,m} 
    &= \Psi_X^{\dagger}(\Psi_X + \Delta t \ \Psi^{'}_X + \bm{o}_{m,N}(\Delta t)) \notag \\
    &= I + \Delta t \ \left(\Psi_X^* \Psi_X\right)^{-1}\left(\Psi_X^* \Psi^{'}_X\right) + \Psi_X^{\dagger}\bm{o}_{m,N}(\Delta t) \notag \\
    % &= I + \Delta t \cdot \widehat{G}^{-1}\widehat{H} + \Psi_X^{\dagger}\bm{o}_{m,N}(\Delta t) \notag \\
    &= \widehat{K}_{N,\Delta t,m} + \Psi_X^{\dagger}\bm{o}_{m,N}(\Delta t),
\end{align}
where $\dagger$ denotes the pseudoinverse. When $\Delta t$ is very small, we omit $\bm{o}_{m,N}(\Delta t)$ and keep $\widehat{K}_{N,\Delta t,m}$ as in the Eq.\eqref{eq: sdmd_kpm_operator_2}.
    
\begin{remark}
    We typically compute $(\widehat{G}+\gamma I)^{-1}$ instead of $\widehat{G}^{-1}$ for some small number $\gamma>0$ to avoid singularity.
\end{remark}

%%%% Section 3.2 %%%%
\subsection{Computing Algorithm}
This section presents the algorithmic implementation of our method. To provide a clear understanding of the computational procedure, we first present a flow chart in Figure \ref{fig:sdmd_flow_chart} that illustrates the key steps of our approach. Following the flow chart, we provide a detailed pseudocode in Algorithm \ref{alg:sdmd} that formalizes the computational steps. The algorithm takes as input the time series data and system parameters, and outputs the approximated Koopman operator. Each step in the algorithm corresponds to the theoretical framework developed in Section \ref{sdmd}, ensuring a complete implementation of our method.

The coefficients \( \mathbf{b}(x) \) and \( \mathbf{\sigma}(x) \) of the stochastic differential equation (SDE) can either be assumed as known for predefined models or estimated from sampled time-series data (see \ref{appendix: sde_coef_estimation}). Specifically, for each independently and identically distributed (i.i.d.) initial point \( x_k \), these coefficients can be approximated using discrete-time methods based on observed trajectories. For instance, \( \mathbf{b}(x) \) can be derived from finite differences to approximate the derivative, while \( \mathbf{\sigma}(x) \) can be inferred from the covariance of the increments.
%%%% Flow Chart of SDMD %%%%
\begin{figure}[htb]
    \centering    
    \includegraphics[width=0.8\linewidth]{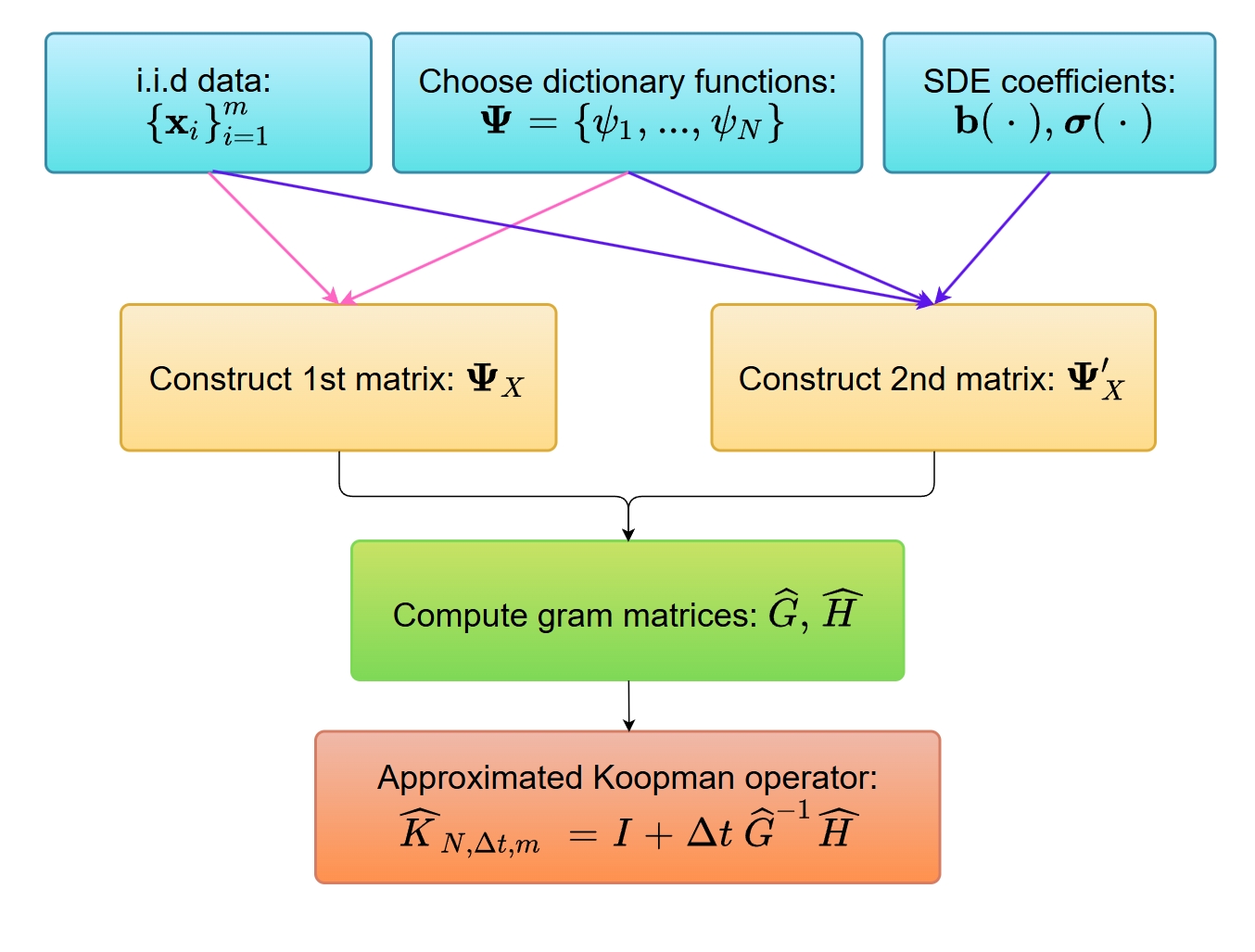}
    \caption{A flow chart for SDMD method.}
    \label{fig:sdmd_flow_chart}
\end{figure}

%%%% Pseudocode of SDMD %%%%
\begin{algorithm}[htb]
\caption{Estimation of stochastic Koopman operator}
\label{alg:sdmd}
\begin{algorithmic}[1]
\Require i.i.d. data $\{ \mathbf{x}_k\}_{k=1}^m$, dictionary functions $\{\psi_1,\ldots,\psi_N\}$, SDE coefficients $\mathbf{b}(\cdot)$, $\mathbf{\sigma}(\cdot)$, sampling time step $\Delta t$, regularization parameter $\gamma$.
\State Construct matrices $\Psi_X, \Psi^{'}_X$ (See Eq.\eqref{data_matrices}).
\State Construct empirical gram matrices $\widehat{G}=(1/m)\Psi_X^*\Psi_X, \widehat{H}=(1/m)\Psi_X^*\Psi^{'}_X$ (See Eq.\eqref{eq: empirical_gram}).
\State Compute $\widehat{K}_{N, \Delta t, m}=I+\Delta t \ \widehat{G}^{-1}\widehat{H}$ (See Eq.\eqref{eq: sdmd_kpm_operator_2}).
\Ensure Approximated Koopman operator $\widehat{K}_{N, \Delta t, m}$.
\end{algorithmic}
\end{algorithm}

\begin{remark}\label{rmk:generalized_eigen}
    Although in our implementation the inverse of $\hat{G}+\gamma I$ in Eq.~\eqref{eq: sdmd_kpm_operator_2} is computed via Cholesky factorization for numerical stability, the spectral computation can also be formulated as a generalized eigenvalue problem. Specifically, we have
    $$
        \widehat{H} v = \frac{\mu-1}{\Delta t}\,\widehat{G}\,v, \quad \lambda = \frac{\mu-1}{\Delta t},
    $$
    where $\mu$ and $v$ are the approximated eigenpair of $\mathcal{K}$, and $\lambda$ corresponds to the approximated eigenvalue of the generator $\mathcal{A}$.
\end{remark}

\subsection{Computing in Deterministic Continuous-Time Systems}
While the SDMD framework is designed with stochastic systems in mind, the core idea of approximating the Koopman semigroup using a truncated Taylor expansion can also be applied to deterministic continuous-time systems. In such cases, the underlying dynamics are governed by an ODE:
$$
    \frac{d}{dt} \mathbf{x}(t) = \mathbf{b}(\mathbf{x}),
$$
and the corresponding Koopman generator becomes a Lie derivative along the vector field $\mathbf{b}(\mathbf{x})$:
$$
    \mathcal{A} f(\mathbf{x}) = \mathbf{b}(\mathbf{x}) \cdot \nabla f(\mathbf{x}).
$$
Then, the Taylor expansion of the semigroup remains:
$
\mathcal{K}^{\Delta t} f(\mathbf{x}) \approx f(\mathbf{x}) + \Delta t \mathcal{A} f(\mathbf{x}),
$
which leads to a similar approximation formula as in Eq.\eqref{eq: sdmd_kpm_operator_2} where $\mathcal{A} \psi_j(\mathbf{x}) = \mathbf{b}(\mathbf{x}) \cdot \nabla \psi_j(\mathbf{x})$. This can be viewed as a continuous-time version to the standard EDMD method \cite{Williams_2015}, which is formulated for discrete-time snapshot pairs, and potentially bridges the gap between EDMD and generator-based methods like gEDMD \cite{KLUS2020132416}. 
\begin{remark}
    It should be noted that the time-difference approach of EDMD and gEDMD may produce significantly different results even in the limit of high sampling rate ($\Delta t \to 0$), particularly when the data sampling measure does not satisfy certain regularity conditions. However, SDMD additionally truncated the higher order term $o(\Delta t)$ from stochastic Taylor expansion to enhance the numerical stability. See \cite[Section~4.3]{bramburger2024auxiliary}, \cite[Section~4.2]{math9192495} and Section~\ref{sec:comparison} for more details.
\end{remark}

% \newpage
%%%% Section 4 %%%%
\section{Convergence Analysis}

This section analyzes the convergence of our stochastic Koopman operator approximation scheme. The analysis follows a sequential framework, where three fundamental regimes are examined in order:

\noindent\textbf{Large Data Convergence}: We first establish the convergence of the empirical approximation by deriving probabilistic error bounds that explicitly quantify the convergence rate as the sample size $m\to\infty$.

\noindent\textbf{Zero-Limit of Sampling Time}: Next we prove that time-discretized approximation converges to the true Koopman generator as the sampling interval $\Delta t\to 0$. This limit builds on large data convergence and connects discrete-time computations to continuous dynamics.

\noindent\textbf{Large Dictionary Size Convergence}: We finally demonstrate that finite-dimensional approximation of the Koopman generator and semigroups converges to their infinite-dimensional counterparts as dictionary size $N\to\infty$. This step completes the convergence analysis.

This structured approach $\lim_{N\to\infty}\lim_{\Delta t \to 0}\lim_{m\to\infty}$ highlights the critical interplay between these limits and ensures rigorous convergence results through appropriate function spaces.

The convergence analysis will be based on the following assumption:
\begin{assumption}\label{ass: dictionry_bound}
    Let $\mathcal{O}\subset\mathcal{M}$ be a compact subset that contains all sampled data points $\mathbf{x}_1, \dots, \mathbf{x}_m$. We assume that for each basis function $\psi_i$, both $\psi_i$ and $\mathcal{A} \psi_i$ are uniformly bounded on $\mathcal{O}$, i.e., there exist constants $C=C_{N,\mathcal{O}}, L=L_{N,\mathcal{O}} > 0$ such that:
    $$
    |\psi_i(x)| \leq C, \quad |\mathcal{A} \psi_i(x)| \leq L \quad \text{for all } x \in \mathcal{O},\; i = 1,\dots,N.
    $$
\end{assumption}

%%%% Section 4.1 %%%%
\subsection{Convergence in the Limit of Large Data}

Fix dictionary size $N$ and sampling time step $\Delta t$. Denote \( K_{N,\Delta t} \) by the matrix that represents the large data limit of \( \tilde{K}_{N,\Delta t,m} \) given in Eq.\eqref{eq: sdmd_kpm_operator}:
\begin{align}\label{eq: K_N_Delta_t}
    K_{N,\Delta t} 
    &= \lim_{m\to\infty} \tilde{K}_{N,\Delta t,m} \notag \\
    % &= \widehat{K}_{N,\Delta t,m} + \lim_{m\to\infty} \Psi_X^{\dagger}\bm{o}_{m,N}(\Delta t) \notag \\
    &= \lim_{m\to\infty} \left( \widehat{K}_{N,\Delta t,m} + \Psi_X^{\dagger}\bm{o}_{m,N}(\Delta t) \right) \notag \\
    &= I + \Delta t \ G^{-1}H + \bm{o}_N(\Delta t),
\end{align}
where $\bm{o}_N(\Delta t) = \lim_{m\to\infty}\Psi_X^{\dagger}\bm{o}_{m,N}(\Delta t)$. Note that, $\bm{o}_N(\Delta t)$ exists since both $K_{N,\Delta t}$ exists from EDMD theory \cite{Williams_2015, Korda_2017} and $\lim_{m\to\infty} \widehat{K}_{N,\Delta t,m} = I + \Delta t \ G^{-1}H$ exists due to Eq.\eqref{eq: empirical_gram_lim}.
\begin{remark}\label{rmk: uniform_cvg}
    Each element in the matrix $\Psi_X^\top \bm{o}_{m,N}(\Delta t)$ is $o(\Delta t)$. After taking the large data limit $m\to\infty$, each element in $\bm{o}_N(\Delta t)$ is still $o(\Delta t)$.
\end{remark}

Now we aim to prove that the empirical Koopman matrix $\widehat{K}_{N,\Delta t,m}$ converges to its large-data limit $K_{N,\Delta t}$ as $m\to\infty$; more specifically, we shall prove a concentration inequality that quantifies the probability of their difference exceeding any given threshold $\epsilon > 0$. Such a probabilistic bound will demonstrate that the empirical approximation becomes increasingly accurate as the sample size $m$ grows. Specifically, we aim to bound the following probabilistic error bound in the following theorem:

%%%% Main Theorem of Section 4.1 %%%%
\begin{theorem}\label{thm: prob_cvg_K}
    Let $\epsilon>0$. Define $\tilde{\epsilon} = \left(\epsilon - \|\bm{o}_N(\Delta t)\|_F\right)/\Delta t$. Then, with same notation and conditions defined in Lemma \ref{lemma: Gramian_estimation} and Lemma \ref{lemma: Philipp2024}, we have
    \begin{equation*}
        \mathbb{P}\left(\Vert \widehat{K}_{N,\Delta t,m} - K_{N,\Delta t} \Vert_F > \epsilon \right) 
        \leq 2N^2 \left[ \exp \left( -\frac{m}{8} \left(\frac{\tilde{\epsilon} \|G^{-1}\|^{-1}}{N^{3/2} C^2 \tau_{\epsilon}}\right)^2 \right) 
        + \exp \left( -\frac{m}{8} \left(\frac{\tilde{\epsilon} \|H\|}{N^{3/2} C \tau_{\epsilon} L}\right)^2 \right) \right],
    \end{equation*}
    where $\|\cdot\|$ denotes the matrix operator norm.
\end{theorem}
We will introduce the following lemmas in order to prove Theorem \ref{thm: prob_cvg_K}. Note that, Lemma \ref{lemma: Hoeffding} is a result of McDiarmid inequality \ref{lemma: McDiarmid_ineq}.
% For deriving such a probabilistic bound, we will employ \textit{Hoeffding's Inequality} , which is a powerful concentration inequality for functions of independent random variables and is essential for its application.

\begin{lemma}[Hoeffding's Inequality \cite{Vershynin_2018}]\label{lemma: Hoeffding}
    Assume $X_1, \ldots, X_n$ are i.i.d. with each $X_i \in [a_i, b_i]$ for all $1\leq i \leq n$. Then, for any $\epsilon>0$,
    % \[
    % \mathbb{P}\left( \sum_{i=1}^n X_i - \mathbb{E}\left[\sum_{i=1}^n X_i\right] \geq \epsilon \right) \leq e^{-\frac{2\epsilon^2}{\sum_{i=1}^n (b_i - a_i)^2}}
    % \]
    % and
    \[
        \mathbb{P}\left(\left| \sum_{i=1}^n X_i - \mathbb{E}\left[\sum_{i=1}^n X_i\right]\right| \geq \epsilon \right) \leq 2e^{-\frac{2\epsilon^2}{\sum_{i=1}^n (b_i - a_i)^2}}.
    \]
\end{lemma}
% \begin{proof}
% Let $f(X_1, \ldots, X_n) = \sum_{i=1}^n X_i$ satisfying the \textit{Bounded Differences Property} given in \ref{McDiarmid's Inequality}, we can apply the McDiarmid's Inequality \cite{ref1} and immediately obtain the result.
% \end{proof}

% \begin{remark}
%     Refer to \ref{McDiarmid's Inequality} for more details on Bounded Differences Property and McDiarmid's Inequality. 
% \end{remark}

\begin{lemma}\label{lemma: Gramian_estimation}
    Suppose Assumption~\ref{ass: dictionry_bound} holds with uniform bounds $C, L > 0$. Then, for any $\epsilon>0$, 
    \[
        \mathbb{P} \left( \left\| \widehat{G} - G \right\|_F \geq \epsilon \right) \leq 2N^2 \exp \left( -\frac{m \epsilon^2}{8 N^2 C^4} \right),
    \]
    \[
        \mathbb{P} \left( \left\| \widehat{H} - H \right\|_F \geq \epsilon \right) \leq 2N^2 \exp \left( -\frac{m\epsilon^2}{8 N^2 C^2 L^2} \right).
    \]
\end{lemma}
\begin{proof}
Let $G, \widehat{G}$ be defined as in Section 2. Define $\eta_{ij}(\mathbf{x}_k) \coloneqq \psi_i(\mathbf{x}_k)\psi_j(\mathbf{x}_k)$. Then, 
\begin{equation*}
    \| \widehat{G} - G \|_F^2 
    = \sum_{i=1}^{N}\sum_{j=1}^{N} \left| \widehat{G}_{ij} - G_{ij} \right|^2 = \sum_{i=1}^{N}\sum_{j=1}^{N} \left| \frac{1}{m}\sum_{k=1}^m \tilde{\eta}_{ij}(\mathbf{x}_k) \right|^2,
\end{equation*}
where $\tilde{\eta}_{ij}(\mathbf{x}_k) \coloneqq \eta_{ij}(\mathbf{x}_k) - \mathbb{E}\left[\eta_{ij}(\mathbf{x}_k)\right] = \eta_{ij}(\mathbf{x}_k) - \mathbb{E}\left[\eta_{ij}(\mathbf{x}_1)\right]$ with $\{\mathbf{x}_k\}_{k=1}^m$ being i.i.d. and thus $\left| \tilde{\eta}_{ij}(\mathbf{x}_k) \right| \leq 2C^2$ for all $1 \leq i,j \leq N$.
Next, applying Hoeffding's inequality from Lemma \ref{lemma: Hoeffding}, we have
\begin{align*}
    \mathbb{P} \left( \left\| \widehat{G} - G \right\|_F \geq \epsilon \right) 
    % &= \mathbb{P} \left( \left\| \widehat{G} - G \right\|_F^2 \geq \epsilon^2 \right) \\
    &\leq \mathbb{P} \left( \sum_{i=1}^{N}\sum_{j=1}^{N} \left| \frac{1}{m} \sum_{k=1}^m \tilde{\eta}_{ij}(\mathbf{x}_k) \right|^2 \geq \epsilon^2 \right) \\
    &\leq \sum_{i=1}^{N}\sum_{j=1}^{N} \mathbb{P} \left( \left| \frac{1}{m} \sum_{k=1}^m \tilde{\eta}_{ij}(\mathbf{x}_k) \right|^2 \geq \epsilon^2/N^2 \right) \\
    &\leq N^2 \mathbb{P} \left( \frac{1}{m} \sum_{k=1}^m \max_{1\leq i,j \leq N}\left| \tilde{\eta}_{ij}(\mathbf{x}_k) \right| \geq \epsilon/N \right) \\
    % &\leq 2N^2 \exp \left( -\frac{2(m\epsilon/N)^2}{m(4C^2)^2} \right) \\
    &\leq 2N^2 \exp \left( -\frac{m \epsilon^2}{8 N^2 C^4} \right).
\end{align*}
Similarly, since 
$|\psi_i(\mathbf{x}_k) \mathcal{A} \psi_j(\mathbf{x}_k)| \leq C L$,
we have
\[ \mathbb{P} \left( \| \widehat{H} - H \|_F \geq \epsilon \right) \leq 2 N^2 \exp \left( -\frac{m\epsilon^2}{8 N^2 C^2 L^2} \right). \]
\end{proof}

\begin{lemma}[Lemma C.5 \cite{philipp2024variancerepresentationsconvergencerates}]\label{lemma: Philipp2024}
    Let $G, H \in \mathbb{R}^{N \times N}$ be such that $G$ is invertible and $H \neq 0$. Let $\widehat{G}, \widehat{H} \in \mathbb{R}^{N \times N}$ be random matrices such that $\widehat{G}$ is invertible a.s. Then for any $\epsilon > 0$ we have
    \[
    \mathbb{P}\left(\|G^{-1}H - \widehat{G}^{-1}\widehat{H}\| > \epsilon \right) \leq \mathbb{P}\left(\|H - \widehat{H}\| > \frac{\epsilon}{\tau_{\epsilon}} \|H\|\right) + \mathbb{P}\left(\|G - \widehat{G}\| > \frac{\epsilon}{\tau_{\epsilon}} \|G^{-1}\|^{-1}\right),
    \]
    where $\tau_{\epsilon} = 2\|G^{-1}\|\|H\| + \epsilon$.
\end{lemma}

\begin{remark}
    We choose operator norm $\|\cdot\|$ here since it is submultiplicative, i.e., $\|GH\|\leq\|G\|\|H\|$ while Frobenious norm $\|\cdot\|_F$ is not.
\end{remark}

\begin{proof}[\textbf{Proof of Theorem \ref{thm: prob_cvg_K}}]
Recall that $\tilde{\epsilon} = \left(\epsilon - \|\bm{o}_N(\Delta t)\|_F\right)/\Delta t$ as in the statement of Theorem \ref{thm: prob_cvg_K},
\begin{align*}
    \mathbb{P}\left(\Vert \widehat{K}_{N,\Delta t,m} - K_{N,\Delta t} \Vert_F > \epsilon \right)
    &= \mathbb{P}\left(\Vert \Delta t \ \widehat{G}^{-1} \widehat{H} - \Delta t \ G^{-1}H + \bm{o}_N(\Delta t) \Vert_F > \epsilon \right) \\
    &\leq \mathbb{P}\left( \Delta t \ \Vert \widehat{G}^{-1} \widehat{H} - G^{-1}H \Vert_F + \|\bm{o}_N(\Delta t)\|_F > \epsilon \right) \\
    &= \mathbb{P}\left( \Vert \widehat{G}^{-1} \widehat{H} - G^{-1}H \Vert_F  > \tilde{\epsilon} \right) \\
    &\leq \mathbb{P}\left( \Vert \widehat{G}^{-1} \widehat{H} - G^{-1}H \Vert  > \tilde{\epsilon}/\sqrt{N} \right) \\
    &\leq \mathbb{P}\left(\|G - \widehat{G}\| > \frac{\tilde{\epsilon}}{\sqrt{N}\tau_{\epsilon}} \|G^{-1}\|^{-1} \right) + \mathbb{P}\left(\|H - \widehat{H}\| > \frac{\tilde{\epsilon}}{\sqrt{N}\tau_{\epsilon}} \|H\| \right) \\
    % &\leq 2N^2 \exp \left( -\frac{m (\frac{\tilde{\epsilon}}{\sqrt{N}\tau_{\epsilon}} \|G^{-1}\|^{-1})^2}{8 N^2 C^4} \right) + 2N^2 \exp \left( -\frac{m (\frac{\tilde{\epsilon}}{\sqrt{N}\tau_{\epsilon}} \|H\|)^2}{8 N^2 C^2 L^2} \right) \\
    &\leq 2N^2 \left[ \exp \left( -\frac{m}{8} \left(\frac{\tilde{\epsilon} \|G^{-1}\|^{-1}}{N^{3/2} C^2 \tau_{\epsilon}}\right)^2 \right) 
    + \exp \left( -\frac{m}{8} \left(\frac{\tilde{\epsilon} \|H\|}{N^{3/2} C \tau_{\epsilon} L}\right)^2 \right) \right].
\end{align*}
\end{proof}

%%%% Section 4.2 %%%%
\subsection{Convergence in the Zero-Limit of Sampling Time}

Define the matrix $A_{N,\Delta t}$ as 
\begin{equation}\label{eq: matrix_generator_cvg}
    A_{N,\Delta t} \coloneqq \frac{K_{N,\Delta t}  - I }{\Delta t} = G^{-1}H + \frac{\bm{o}_N(\Delta t)}{\Delta t},
\end{equation}
where we use Eq.\eqref{eq: K_N_Delta_t} for $K_{N,\Delta t}$ in the last equality; and let $\mathcal{A}_{N,\Delta t}: \mathcal{F}_N \to \mathcal{F}_N$ be the linear map defined by 
$$
    (\psi_1, \ldots, \psi_N) \, \mathbf{a} \mapsto (\psi_1, \ldots, \psi_N)\, A_{N,\Delta t} \, \mathbf{a},
$$
for any coefficient vector \(\mathbf{a} \in \mathbb{R}^N\). 

Let $\mathcal{A}_N \coloneqq \mathcal{P}_N \mathcal{A} \mathcal{P}_N$ be the finite dimensional approximation of Koopman generator $\mathcal{A}$ where we denote $\mathcal{P}_N$ by the orthogornal projection from $\mathcal{F}$ to $\mathcal{F}_N$ and $\mathcal{F}_N = \text{span} \{\psi_1, \dots, \psi_N \}$ as mentioned in last section. 

In the following Theorem \ref{thm: cvg_zero_time}, we will show that this time-discretized approximant $\mathcal{A}_{N,\Delta t}$ converges to $\mathcal{A}_N$ in operator norm as $\Delta t \to 0$.

%%%% Main Theorem in 4.2 %%%%
\begin{theorem}\label{thm: cvg_zero_time}
    Let $\mathcal{A}_{N,\Delta t}$ and $\mathcal{A}_N$ be defined as above. For each dictionary size $N>0$, we have
    $$
        \lim_{\Delta t \to 0} \|\mathcal{A}_{N,\Delta t} - \mathcal{A}_N\| = 0.
    $$
\end{theorem}

%%%% Proof of Main Theorem in 4.2 %%%%
\begin{proof}[\textbf{Proof of Theorem \ref{thm: cvg_zero_time}}]
%%%%%%%%%%%%%%%%%%%%%%%%%%%%%%
Define \( A_N \coloneqq G^{-1}H \) where matrices $G, H$ are given in Eq.\eqref{eq: empirical_gram_lim}. Applying Galerkin approximation for the Koopman generator $\mathcal{A}$, we have that \( A_N \) is the matrix representation of \(\mathcal{A}_N\) on $\mathcal{F}_N$ \cite[Proposition 3.5]{KLUS2020132416}, i.e., 
\[
    \mathcal{A}_N\bigl((\psi_1, \dots, \psi_N)\mathbf{a}\bigr) = (\psi_1, \dots, \psi_N)A_N\mathbf{a},
\]    
for any coefficient vector \(\mathbf{a} \in \mathbb{R}^N\). From Eq.\eqref{eq: matrix_generator_cvg}, we know that
\begin{align*}
    A_{N,\Delta t} 
    % &= \frac{K_{N,\Delta t}  - I }{\Delta t} \\
    % &= \frac{ I + \Delta t \cdot G^{-1}H + \bm{o}_N(\Delta t) - I }{\Delta t} \\
    % &= G^{-1}H + \bm{o}_N(\Delta t)/\Delta t \\
    &= A_N + \bm{o}_N(\Delta t)/\Delta t.
\end{align*}
Since each element in $\bm{o}_N(\Delta t)$ is $o(\Delta t)$ as pointed out in Remark \ref{rmk: uniform_cvg}, we have $\lim_{\Delta t \to 0} \|A_{N,\Delta t} - A_N\| = \lim_{\Delta t \to 0} \|\bm{o}_N(\Delta t)\|/\Delta t = 0$. The conclusion of the theorem follows.
\end{proof}

\begin{remark}
    In finite dimensional space, the uniform operator topology is equivalent to strong operator topology.
\end{remark}

\subsection{Convergence in the Limit of Large Dictionary Size}\label{cvg_dic_size}

This section establishes the following two convergence results of our approximation scheme for stochastic Koopman operator as the dictionary size $N\to\infty$. 

\textbf{Generator Convergence }
Under Assumption \ref{ass: projection}, the finite dimensional generator $\mathcal{A}_N = \mathcal{P}_N\mathcal{A}\mathcal{P}_N$ strongly converges to the true generator $\mathcal{A}$ strongly as $N \to \infty$.

\textbf{Semigroup Convergence }
Under Assumption \ref{ass: core}, the approximated semigroups $(e^{t\mathcal{A}_N})_{t\geq 0}$ strongly converge to the true Koopman semigroup $(e^{t\mathcal{A}})_{t\geq 0}$ uniformly over compact time interval as $N\to\infty$.

The proof builds upon the established framework in \cite[Section 4]{llamazareselias2024datadrivenapproximationkoopmanoperators} and utilizes the Trotter-Kato Approximation theorem \cite{engel1999one}, which will be introduced systematically in the following subsections.

%%%% Section 4.3.1 %%%%
\subsubsection{Convergence of Finite Dimensional Koopman Generator}\label{cvg_pj_kpm_generator}

By Assumption \ref{ass:str_cont_semigroup_K} in Section \ref{sec: intro_stochastic_kpm_operator}, we define the inner product
\[
    \langle f, g \rangle_{\mathcal{A}} := \langle f, g \rangle_{\rho} + \langle \mathcal{A} f, \mathcal{A} g \rangle_{\rho}, \quad \forall f, g \in \mathcal{D}(\mathcal{A}),
\]
and the corresponding graph norm
\[
    \| f \|_{\mathcal{A}} := \sqrt{ \langle f, f \rangle_{\mathcal{A}}}.
\]
Clearly, the generator 
% $\mathcal{A}: (\mathcal{D}(\mathcal{A}), \| \cdot \|_{\mathcal{A}}) \to (\mathcal{F}, \| \cdot \|_{\rho})$ 
satisfies $\| \mathcal{A} \|_{\mathcal{D}(\mathcal{A})\rightarrow\mathcal{F}} \coloneqq \sup_{\|f\|_{\mathcal{A}}=1}\|\mathcal{A}f\|_{\rho} \leq 1$.

Recall that $\mathcal{P}_N$ is projection of $\mathcal{F}$ onto $\mathcal{F}_N$ equipped with $\| \cdot \|_{\rho}$. Denote by $\widetilde{\mathcal{P}}_N$ the projection of $\mathcal{D}(\mathcal{A})$
onto $\mathcal{F}_N$. We now introduce the following assumptions:
\begin{assumption}\label{ass: projection}
    We assume that the dictionary $\Psi$ satisfies the following conditions: 
    \begin{itemize}
        \item $\lim_{N \to \infty} \| (\mathcal{P}_N - \mathrm{I}) f \|_{\rho} = 0, \quad \forall f \in \mathcal{F}$. 
        \item $\lim_{N \to \infty} \| (\widetilde{\mathcal{P}}_N - \mathrm{I}) f \|_{\mathcal{A}} = 0, \quad \forall f \in \mathcal{D}(\mathcal{A})$,
    \end{itemize}
\end{assumption}
\begin{remark}
    The second part of Assumption~\ref{ass: projection} is moderately more restrictive than first part, as it requires convergence in the graph norm of the generator. This assumption is standard in the Galerkin approximation of unbounded operators (e.g., see \cite[Assumption 4]{llamazareselias2024datadrivenapproximationkoopmanoperators}) but may not be satisfied for arbitrary dictionaries. However, in practice, data are typically collected in a compact subset $\mathcal{O} \subset \mathcal{M}$, so the restriction to $\mathcal{O}$ makes this assumption much less limiting. On such compact domain, this condition can often be approximated by selecting a sufficiently smooth dictionary.
\end{remark}
\noindent The next result shows the strong convergence of $\mathcal{A}_N$ to $\mathcal{A}$ as $N\to\infty$.
%%%% Main Theorem in Section 4.3.1 %%%%
\begin{theorem}\label{thm: generator_strong_convergence}
    Suppose Assumption \ref{ass: projection} holds. Then for all $f \in \mathcal{D}(\mathcal{A)}$,
    \begin{equation*}
        \lim_{N\to\infty} \|\mathcal{A}_N f - \mathcal{A} f\|_\rho = 0.
    \end{equation*}
\end{theorem}

%%%% Proof of Theorem 4.10 %%%%
\begin{proof}[\textbf{Proof of Theorem \ref{thm: generator_strong_convergence}}]
For any $f \in \mathcal{D}(\mathcal{A})$, let $\widetilde{f}_N \coloneqq \widetilde{\mathcal{P}}_Nf$. Then, 
\begin{align*}
    \|\mathcal{A}_N f - \mathcal{A} f\|_\rho 
    &= \|\mathcal{P}_N \mathcal{A} \mathcal{P}_N f - \mathcal{A} f\|_\rho \\
    &= \|\mathcal{P}_N \mathcal{A} \mathcal{P}_N (f - \widetilde{f}_N) + \mathcal{P}_N \mathcal{A} \mathcal{P}_N \widetilde{f}_N - \mathcal{A}f \|_{\rho} \\
    &\leq \underbrace{\|\mathcal{P}_N \mathcal{A} \mathcal{P}_N (f - \widetilde{f}_N)\|_{\rho}}_{\text{first term}} + \underbrace{\|\mathcal{P}_N \mathcal{A} \mathcal{P}_N \widetilde{f}_N - \mathcal{A} f \|_{\rho}}_{\text{second term}} .
\end{align*}

The first term $\|\mathcal{P}_N \mathcal{A} \mathcal{P}_N (f - \widetilde{f}_N)\|_{\rho} \to 0$ as \(N\to\infty\) follows immediately from the second part of Assumption~\ref{ass: projection}. The convergence of the second term is given by the following:
\begin{align*}
    \|\mathcal{P}_N \mathcal{A} \mathcal{P}_N \widetilde{f}_N - \mathcal{A}f \|_{\rho} 
    &= \|\mathcal{P}_N \mathcal{A} \widetilde{\mathcal{P}}_N f - \mathcal{A}f \|_{\rho} \\
    &= \|(\mathcal{P}_N-I+I) \mathcal{A} \widetilde{\mathcal{P}}_Nf - \mathcal{A}f \|_{\rho} \\
    &\leq \|(\mathcal{P}_N-I) \mathcal{A}\widetilde{\mathcal{P}}_Nf\|_{\rho} + \|\mathcal{A} \|_{\mathcal{D}(\mathcal{A})\rightarrow\mathcal{F}} \|\widetilde{\mathcal{P}}_N f - f\|_{\mathcal{A}} \\
    &\to 0 \quad \text{as } N \to \infty,
\end{align*}
where we use Assumption~\ref{ass: projection} and $\|\mathcal{A} \|_{\mathcal{D}(\mathcal{A})\rightarrow\mathcal{F}}\leq1$ in the last limit.
\end{proof}

%%%% Section 4.3.2 %%%%
\subsubsection{Convergence of Finite Dimensional Koopman Semigroups}\label{cvg_pj_kpm_semigroup}

Semigroups generated by $\mathcal{A}_{N}$ and $\mathcal{A}$ are denoted by:
$\mathcal{K}_N^t = e^{t\mathcal{A}_N}$ and $\mathcal{K}^t = e^{t\mathcal{A}}$ respectively.
To establish the convergence of $\mathcal{K}_N^t$ to $\mathcal{K}^t$ as $N\to\infty$, we first introduce the following assumptions:
% \begin{definition}[Definition 1.6 \cite{engel1999one}]
%     A linear subspace $\mathcal{D} \subseteq \mathcal{D}(\mathcal{A})$ is called a \textit{core} for the Koopman generator $\mathcal{A}$ if $\mathcal{D}$ is dense in $\mathcal{D}(\mathcal{A})$ with respect to the graph norm $\| \cdot \|_{\mathcal{A}}$.
% \end{definition}

\begin{assumption}\label{ass: core}
There exists a core $\mathcal{D}$ for $\mathcal{A}$, that is a linear subspace $\mathcal{D} \subseteq \mathcal{D}(\mathcal{A})$ and dense in $\mathcal{D}(\mathcal{A})$ with respect to $\| \cdot \|_{\mathcal{A}}$ \cite[Definition 1.6]{engel1999one} such that:
\begin{itemize}
    \item $\mathcal{D} \subseteq \mathcal{D}(\mathcal{A}_N)$, for all $N \in \mathbb{N}$,
    \item $\forall f \in \mathcal{D}$, $\mathcal{A}_N f \to \mathcal{A} f$ in $\mathcal{F}$ as $N \to \infty$.
\end{itemize}
\end{assumption}

Now we establish the strong convergence of $\mathcal{K}_N^t$ to $\mathcal{K}^t$ uniformly over compact time interval as $N\to\infty$ based on Theorem \ref{thm: generator_strong_convergence} and Assumption \ref{ass: core} in the following result:

%%%% Main theorem in 4.3.2 %%%%
\begin{theorem}\label{thm: result_of_TK}
There exists some constants $D \geq 1$ and $\omega \in \mathbb{R}$ such that the semigroups satisfy
\[
    \|\mathcal{K}_N^t\|, \|\mathcal{K}^t\| \leq De^{\omega t} \quad \text{for all } t \geq 0, N \in \mathbb{N}.
\]
Furthermore, for every $f \in \mathcal{F}$ and every $T > 0$, we have
\[
    \lim_{N \to \infty} \sup_{t \in [0,T]} \|\mathcal{K}_N^t f - \mathcal{K}^t f\|_{\rho} = 0.
\]
\end{theorem}

\begin{proof}[\textbf{Proof of Theorem \ref{thm: result_of_TK}}]
    Since $\mathcal{K}_N^t$ are generated by a finite dimensional operator $\mathcal{A}_N$ for each $N\in\mathbb{N}$, they are strongly continuous; and also $\mathcal{K}^t$ is strongly continuous by Assumption \ref{ass:str_cont_semigroup_K}. Thus, the exponential boundedness in the first part is proved according to \cite[Theorem 2.2]{pazy2012semigroups}. The second part of the theorem is an immediate consequence of Trotter--Kato Approximation Theorem \cite{engel1999one}. 
\end{proof}

\begin{remark}
    The Trotter--Kato Approximation Theorem \cite{engel1999one} ensures uniform semigroup convergence on compact time intervals once strong resolvent convergence of the approximate generators is established, i.e.,
    $(\lambda-\mathcal{A}_N)^{-1}f \to (\lambda-\mathcal{A})^{-1}f
    \quad \text{for all } f\in\mathcal{F}, \ \lambda>\omega,$    
    together with other properties and assumptions. From our generator convergence result in the first part of the theorem, the required resolvent convergence follows directly, hence the semigroup convergence in the second part. See \cite[Section~4]{engel1999one} for details.
\end{remark}

\begin{remark}
    In our setting, $\mathcal{K}^{\Delta t}$ approximates the one-step Koopman operator $e^{\Delta t \mathcal{A}}$. Iterating this operator computes a discrete-time semigroup $(\mathcal{K}^{\Delta t})^n$, which converges to the continuous semigroup $e^{tA}$ as $\Delta t \to 0$. This convergence can be formally justified by the Trotter–Kato theorem under standard assumptions on the generator approximation.
\end{remark}

% This section presents a complete analysis of the convergence properties for finite-dimensional approximation of the Koopman semigroup in stochastic systems. Building upon the concept of a core for the Koopman generator and establishing appropriate boundedness conditions, it is proved that the approximated semigroups $\mathcal{K}^t_N = e^{t\mathcal{A}_N}$ converge strongly to the true Koopman semigroup $\mathcal{K}^t = e^{t\mathcal{A}}$ uniformly over compact time intervals. The pivotal result, Theorem \ref{thm: result_of_TK} demonstrates that this semigroup convergence is guaranteed under suitable conditions, which extends beyond generator convergence, to ensure that the time-evolved dynamics are accurately captured.

%%%% Section 5 %%%%
\section{SDMD with Dictionary Learning (SDMD-DL)}\label{sec:sdmd_dl}

Manual selection of basis functions often requires domain expertise and prior knowledge of the system's dynamics, which can be challenging for complex stochastic systems. Additionally, hand-crafted dictionaries may not optimally capture the underlying state space geometry and lead to suboptimal approximation of the Koopman operator. To address these limitation, SDMD can be integrated with dictionary learning technique \cite{1614066, li2017extended} that automatically extracts optimal basis functions directly from data \cite{li2017extended, liu2024physics, lusch2018deep, 10.1063/5.0157763, xu2025reskoopnetlearningkoopmanrepresentations}. 

%%%% Section 5.1 %%%%
\subsection{Methodology and Discussion}
\textbf{Overall Framework and Design}
In SDMD-DL framework, the matrix \( \Psi_X + \Delta t \ \Psi^{'}_X + \bm{o}_{m,N}(\Delta t) \) in Eq.\eqref{eq: min_loss} is replaced by the data matrix \( \Psi_Y \) and have the following minimization problem
\begin{equation}\label{min_loss_nn}
    \min_{\tilde{K}_{N,\Delta t,m}\in\mathbb{R}^{N\times N}} \| \Psi_Y - \Psi_X \tilde{K}_{N,\Delta t,m} \|_F^2,
\end{equation}
where the matrix \( \Psi_Y \) is defined as
\[
    \Psi_Y \coloneqq 
    \begin{bmatrix}
        \psi_1(\mathbf{y}_1) & \cdots & \psi_{N}(\mathbf{y}_1) \\
        \vdots & \ddots & \vdots \\
        \psi_1(\mathbf{y}_m) & \cdots & \psi_{N}(\mathbf{y}_m)
    \end{bmatrix},
\]
and \( \mathbf{y}_i \) represents the evolved state \( \Delta t \) time after $\mathbf{x}_i$ under stochastic dynamics for each $1\leq i \leq m$, in other words, we can say that the state \( \mathbf{y}_i \) is the realization of this stochastic evolution as given in Eq.\eqref{eq: sde} starting from the initial state \( \mathbf{x}_i \) over \( \Delta t \) time. In SDMD-DL, a neural network is integrated to parameterize dictionary functions \( \bm{\Psi}(\mathbf{x}; \theta) \) for data-driven learning. This approach alternates between optimizing the approximated Koopman operator \( \widehat{K}_{N,\Delta t,m}(\theta) \) by Eq.\eqref{eq: sdmd_kpm_operator_2} and updating the neural network parameters \( \theta \) using stochastic gradient descent for Eq.\eqref{min_loss_nn}, ensuring that both the dictionary functions and the operator approximation improve iteratively.

\noindent\textbf{Parameterization}
We now show more details of how to paramterize the basis functions and the training scheme. Denote \( \bm{\Psi}(\mathbf{x}; \theta) = [\psi_1(\mathbf{x}; \theta), \dots, \psi_N(\mathbf{x}; \theta)]^\top \) by a dictionary parameterized by a neural network, where \( \theta \) represents the trainable parameters. Both \( \Psi_X = \Psi_X(\theta) \) and \( \Psi_Y = \Psi_Y(\theta) \) are computed directly from the neural network and the Koopman operator \( \widehat{K}_{N,\Delta t,m}(\theta) \) is then computed by Eq.\eqref{eq: sdmd_kpm_operator_2} as following
\[
    \widehat{K}_{N,\Delta t,m}(\theta) = I + \Delta t \ \widehat{G}(\theta)^{-1}\widehat{H}(\theta).
\]

The training process involves minimizing a loss function that balances the approximation quality of \( \widehat{K}_{N,\Delta t,m} \) with regularization. Here the loss function is defined as:
\[
    J(\theta, \widehat{K}_{N,\Delta t,m}) = \| \Psi_Y(\theta) - \Psi_X(\theta) \widehat{K}_{N,\Delta t,m}(\theta) \|_F^2 + \gamma R(\widehat{K}_{N,\Delta t,m}),
\]
where $\gamma>0$ is some small positive number and \( R(\widehat{K}_{N,\Delta t,m}) \) is the regularization term. In this work, we use Tikhonov regularization \( R(\widehat{K}_{N,\Delta t,m}) = \|\widehat{K}_{N,\Delta t,m}(\theta)\|_F^2 \). We show the pseudocode in Algorithm \ref{algorithm: sdmd_dl}.

\begin{algorithm}[!htb]
\caption{Estimation of stochastic Koopman operator with dictionary learning}\label{algorithm: sdmd_dl}
\begin{algorithmic}[1]
\Require i.i.d. initial data points $\{\mathbf{x}_k\}_{k=1}^m$ and their $n$-step time series trajectories $\{\{\mathbf{y}_k^{(j)}\}_{j=1}^n\}_{k=1}^m$, dictionary size $N$, sampling time step $\Delta t$, regularization parameter $\gamma>0$, learning rate $\eta>0$, number of epochs $T>0$.
\State Estimate SDE coefficients $\mathbf{b}(\cdot)$ and $\bm{\sigma}(\cdot)$ from the time series data (See \ref{appendix: sde_coef_estimation}).
\State Initialize neural network parameters $\theta$.
\State Initialize Koopman operator $\widehat{K}_{N,\Delta t,m}$.
\For{epoch = 1 to $T$}
    \State Compute $\Psi_X(\theta),\Psi_Y(\theta),\Psi^{'}_X(\theta)$ using $\{(\mathbf{x}_k, \mathbf{y}_k^{(j)})\}_{k=1,j=1}^{m,n}$, $\mathbf{b}(\cdot)$, and $\bm{\sigma}(\cdot)$.
    \State Compute Gram matrices:
        \State \qquad $\widehat{G}(\theta) = \frac{1}{m}\Psi_X(\theta)^*\Psi_X(\theta)$,
        \State \qquad $\widehat{H}(\theta) = \frac{1}{m}\Psi_X(\theta)^*\Psi^{'}_X(\theta)$.
    \State Compute loss:
        \State \qquad $J(\theta,\widehat{K}_{N,\Delta t,m}) = \|\Psi_Y(\theta) - \Psi_X(\theta)\widehat{K}_{N,\Delta t,m}\|_F^2 + \gamma\|\widehat{K}_{N,\Delta t,m}\|_F^2$.
    \State Update neural network parameters:
        \State \qquad $\theta \leftarrow \theta - \eta\nabla_\theta J(\theta,\widehat{K}_{N,\Delta t,m})$.
    \State Update approximated Koopman operator:
        \State \qquad $\widehat{K}_{N,\Delta t,m} = I + \Delta t \ \left(\widehat{G}(\theta)+\gamma I\right)^{-1}\widehat{H}(\theta)$.
\EndFor
\Ensure Approximated Koopman operator $\widehat{K}_{N,\Delta t,m}(\theta)$.
\end{algorithmic}
\end{algorithm}

% \begin{remark}
%     For the estimation of SDE coefficients, see \ref{appendix: sde_coef_estimation}.
% \end{remark}

%%%% Section 5.2 %%%%
\subsection{Comparison with Other Methods}\label{sec:comparison}

First we compare SDMD-DL with EDMD-DL \cite{li2017extended}, which extends traditional EDMD by replacing manually selected dictionaries with ones learned from data using a simple feedforward neural network. While SDMD-DL shares similarities with EDMD-DL in the loss function computing formula, a key distinction lies in how \( \widehat{K}_{N,\Delta t,m} \) is updated. SDMD-DL employs Eq.\eqref{eq: sdmd_kpm_operator_2}, leveraging the stochastic Taylor expansion and incorporating the sampling time \( \Delta t \). In contrast, EDMD-DL uses \( \widehat{K}_{N,\Delta t,m} = (\Psi_X^* \Psi_X)^{-1}(\Psi_X^* \Psi_Y) \), which is better suited for deterministic systems without consideration of stochasticity.

Next, we can also compare our method to gEDMD with dictionary learning setting, i.e., gEDMD-DL, which is almost same as the EDMD-DL's framework except that the loss function in gEDMD-DL is defined to minimize the linear regression error associated with the generator \( \mathcal{A} \) instead of semigroup $\mathcal{K}^{\Delta t}$ as in EDMD-DL. This contains the computation of \( \Psi_X^* \Psi_X \) and \( \Psi_X^* \Psi^{'}_X \), where \( \Psi^{'}_X \) involves evaluating Jacobian and Hessian matrices as given in Eq.\eqref{data_matrices}. However, calculation of \( \Psi^{'}_X \) can be extremely computationally expensive, especially for large datasets, as they require repeatedly computing higher-order derivatives during the Automatic Differentiation process and cross validation process. In contrast, SDMD-DL defines its loss function directly using \( \Psi_Y \), the parameterized dictionary applied to the evolved data, instead of \( \Psi^{'}_X \). By directly utilizing \( \Psi_Y \), SDMD-DL avoids the heavy evaluation of \( \Psi^{'}_X \), thus significantly reducing computational overhead. Furthermore, the use of \( \Psi_Y \) ensures consistency with the stochastic Koopman operator's evolution, which can enable the approximation of the Koopman operator in stochastic systems more accurately.

The following Table \ref{tab:comparison_test} gives a comparison of each method's distinction in the aspects of both loss function and updating formula:
%%%% Table %%%%
\begin{table}[!htb]
    \centering
    \renewcommand{\arraystretch}{1.5}
    \begin{tabular}{|>{\centering\arraybackslash}m{2cm}|m{12cm}|}
        \hline
        \centering
        \textbf{EDMD-DL} & \textbf{Loss function:}
        % \vspace{0.1cm}
        \[ 
            J(\theta, \widehat{K}_{N,\Delta t,m}) = \| \Psi_Y(\theta) - \Psi_X(\theta) \widehat{K}_{N,\Delta t,m}(\theta) \|_F^2 + \gamma R(\widehat{K}_{N,\Delta t,m}) 
        \] 
        % \vspace{-0.2cm}
        \textbf{Updating formula:} \vspace{0.1cm}
        \[ 
            \widehat{K}_{N,\Delta t,m}(\theta) = \Big(\Psi_X(\theta)^* \Psi_X(\theta) + \gamma I\Big)^{-1}\Big(\Psi_X(\theta)^* \Psi_Y(\theta)\Big) 
        \]
        \vspace{-0.2cm}
        \\
        \hline
        \centering
        \textbf{gEDMD-DL} & \textbf{Loss function:}
        % \vspace{0.1cm}
        \[ 
            J(\theta, \widehat{K}_{N,\Delta t,m}) = \| \Psi^{'}_X(\theta) - \Psi_X(\theta) \widehat{A}_{N,m}(\theta) \|_F^2 + \gamma R(\widehat{A}_{N,m}) 
        \]
        % \vspace{-0.2cm}
        \textbf{Updating formula:} 
        % \vspace{0.1cm}
        \[ 
            \widehat{A}_{N,m}(\theta) = \Big(\Psi_X(\theta)^* \Psi_X(\theta) + \gamma I\Big)^{-1}\Big(\Psi_X(\theta)^* \Psi^{'}_X(\theta)\Big) 
        \]
        \vspace{-0.2cm}
        \\
        \hline
        \centering
        \textbf{SDMD-DL} & \textbf{Loss function:} 
        % \vspace{0.1cm}
        \[ 
            J(\theta, \widehat{K}_{N,\Delta t,m}) = \| \Psi_Y(\theta) - \Psi_X(\theta) \widehat{K}_{N,\Delta t,m}(\theta) \|_F^2 + \gamma R(\widehat{K}_{N,\Delta t,m}) 
        \]
        % \vspace{-0.2cm}
        \textbf{Updating formula:} 
        % \vspace{0.1cm}
        \[ 
            \widehat{K}_{N,\Delta t,m}(\theta) = I + \Delta t \  \Big(\Psi_X(\theta)^* \Psi_X(\theta) + \gamma I\Big)^{-1}\Big(\Psi_X(\theta)^* \Psi^{'}_X(\theta)\Big)
        \] 
        \vspace{-0.2cm}
        \\
        \hline
    \end{tabular}
    \vspace{0.5cm}
    \caption{Comparison test between EDMD-DL, gEDMD-DL and SDMD-DL.}
    \label{tab:comparison_test}
\end{table}

%%%% Section 6 %%%%

\section{Application}
This section evaluates the proposed Stochastic Dynamic Mode Decomposition (SDMD) framework through four representative experiments, each selected to highlight unique features of stochastic dynamical systems and demonstrate SDMD’s effectiveness. We first provide a brief overview of these examples and explain our objectives and accomplishments in applying SDMD. These experiments showcase the method’s ability to approximate the Koopman semigroup and their spectral properties which offers insights into system behaviors such as oscillatory dynamics, decay rates, and metastable transitions, while also comparing SDMD’s performance to established methods like EDMD and gEDMD.

The 2D Stuart-Landau system \cite{tantet2020ruellepollicottresonancesstochasticsystems} is a well-known model in nonlinear dynamics, often used to study oscillatory behaviors perturbed by stochastic noise. It behaves as a limit cycle under some conditions whose radius and phase evolve under both deterministic and random influences, which makes it an excellent testbed for analyzing spectral properties in the presence of noise. In this experiment, we apply SDMD to capture the eigenvalues of the Koopman semigroup, aiming to accurately reflect the system’s oscillatory patterns and phase diffusion. Our results demonstrate SDMD’s superior ability to handle stochastic perturbation compared to traditional method like EDMD.

The 1D Ornstein-Uhlenbeck (OU) process \cite{gillespie1996exact} is a time-reversible system characterized by mean reversion driven by a linear drift and Gaussian noise, with closed-form solution for its spectral properties. Here, we use SDMD to estimate the leading eigenpairs of the Koopman semigroup and achieve high precision that aligns with theoretical expectation.

The 2D Triple-Well system \cite{schutte2013metastability} models a metastable stochastic process with a potential landscape and is commonly studied in chemical physics and molecular dynamics. Its dynamics exhibit slow transitions between basins alongside rapid fluctuation within them, which arised a challenge for capturing long-term behaviors. With SDMD, we identify the slow timescales of basin transitions through the Koopman semigroup’s eigenvalues and eigenfunctions, which successfully reveals the system’s metastable structure. This experiment shows SDMD’s strength in analyzing complex, time-reversible metastable systems and its advantage over methods like gEDMD in computational efficiency and accuracy.

The final example involves a minimal two-dimensional neural-mass model \cite{montbrio2015macroscopic} that exhibits multiscale dynamics, including a slow-modulating latent input, an intermediate transient decay, and fast oscillations. Using only the observable state variables, SDMD accurately recovers eigenfunctions corresponding to these latent timescales, even in the presence of stochastic perturbation. In particular, SDMD reveals one eigenfunction that reflects the slow switching of input-modulated regimes, and another that captures transient dynamics during transitions. In contrast, EDMD with dictionary learning fails to isolate these features, which produces eigenfunctions corrupted by fast oscillations and noise. This example highlights SDMD’s capability to infer hidden multiscale structures from limited, noisy measurements, which can demonstrate its robustness and interpretability in realistic settings.

The four experiments use the Euler-Maruyama (EM) method for numerical integration. In the two examples of OU-process and triple-well system, we observe that the approximated eigenvalues are purely real, consistent with their time-reversible property. We also note that the drift and diffusion coefficients used in the 2D Stuart-Landau system are analytical values in polar coordinates, whereas those in the 1D OU process and 2D Triple-Well system are estimated from collected data using a simple neural network (NN) instead of manually selecting basis functions. The error bound analysis of this coefficient estimation method is provided in \cite{gu2023stationary}. The following subsections detail the setups and results. All experiments are available on our \href{https://github.com/TalkingDoll/SDMD/tree/main}{GitHub} repository.

%%%% Section 6.1 2D Stuart-Landau equation %%%%
\subsection{2D Stuart-Landau Equation}
The 2D stochastic Stuart-Landau (SL) equation \cite{tantet2020ruellepollicottresonancesstochasticsystems} serves as a canonical example to study nonlinear oscillatory dynamics under stochastic perturbation. It is frequently used to validate numerical methods for estimating the stochastic Koopman operator since it has analytical Koopman eigenpairs expression. The equation in standard Cartesian coordinates is the following
\[
    \begin{aligned}
        dx &= \left[ \left( \delta - \kappa(x^2 + y^2) \right) x - \left( \gamma - \beta(x^2 + y^2) \right) y \right] \, dt + \epsilon \, dW_x, \\
        dy &= \left[ \left( \gamma - \beta(x^2 + y^2) \right) x + \left( \delta - \kappa(x^2 + y^2) \right) y \right] \, dt + \epsilon \, dW_y,
    \end{aligned}
\]
where \( R \coloneqq \sqrt{\delta/\kappa} \) is the radius of the limit cycle, $\gamma>0$ is the rotation frequency parameter that controls the linear rotational motion of the system and $\beta>0$ is the nonlinear frequency parameter that controls amplitude-related frequency changes. The noise terms \( dW_x \) and \( dW_y \) are independent Wiener processes with intensity \( \epsilon \). Note that, a positive parameter \( \delta > 0 \) indicates that the system exhibit a stable limit cycle behavior. When $\delta < 0$, the system experiences negative radial growth, with trajectories spiraling inward toward the origin and the origin becomes a stable focus. 

To simplify the analysis, the system is often transformed into polar coordinates
\[
\begin{aligned}
    dr &= \left( \delta r - \kappa r^3 + \frac{\epsilon^2}{2r} \right) \, dt + \epsilon \, dW_r, \\
    d\theta &= \left( \gamma - \beta r^2 \right) \, dt + \frac{\epsilon}{r} \, dW_\theta,
\end{aligned}
\]
where \( r \) and \( \theta \) represent the radius and angular position, respectively, and \( dW_r \) and \( dW_\theta \) are derived Wiener processes. The analytical eigenvalues of the stochastic Stuart-Landau system's Kolmogorov operator, i.e., the Koopman generator, are given as:
\begin{equation}\label{eq: sl_evalue}
    \lambda_{ln} =
    \begin{cases}
        -\frac{n^2 \epsilon^2}{2R^2} + i n (1 - \delta) + O(\epsilon^3), & l = 0, \, n \in \mathbb{Z}, \\
        -2l \delta + i n (1 - \delta) + O(\epsilon), & l \in \mathbb{Z}^+ , \, n \in \mathbb{Z}.
    \end{cases}
\end{equation}
    
% Here, \( l = 0 \) refers to the angular dynamics along the limit cycle, where \( n \) indexes the angular harmonics. The real part of the eigenvalues for \( l = 0 \) describes the phase diffusion rate, proportional to \( -n^2 \epsilon^2 / R^2 \), caused by noise. The imaginary part, proportional to \( n(1 - \delta) \), corresponds to the angular oscillations around the limit cycle. For \( l \in \mathbb{Z}^+ \), \( l \) represents the degree of the Hermite polynomial in the eigenfunction associated with the eigenvalues. Each \( l \) corresponds to a distinct radial mode, with eigenvalues having a real part proportional to \( -2l \delta \) and reflecting the decay rates of radial perturbations. \( n \in \mathbb{Z} \) represents the azimuthal mode number, capturing angular harmonics. The imaginary part of the eigenvalues, proportional to \( n(1 - \delta) \), reflects the oscillatory behavior influenced by the system's angular frequency. These eigenvalues provide a rigorous benchmark to validate numerical methods for the Koopman generator, demonstrating both radial and angular contributions to the system's dynamics.
Here \( l \) and \( n \) are indices that label the radial and angular modes of the system’s dynamics, respectively. We will now talk about the two separate cases of the eigenvalues.

For \( l = 0 \), the focus is on angular dynamics along the limit cycle, with no radial nodes in the eigenfunctions. Here, \( n \in \mathbb{Z} \) indexes the angular frequency, representing the azimuthal mode number, e.g., \( e^{i n \theta} \), where \(\theta\) is the phase angle. The real part \(-\frac{n^2 \epsilon^2}{2R^2}\) of the eigenvalues describes the phase diffusion rate induced by noise (controlled by \(\epsilon\)). This negative term, proportional to \( n^2 \), indicates that higher angular modes (larger \( |n| \)) decay faster due to stochastic dynamics. The imaginary part \( i n (1 - \delta) \) corresponds to the angular oscillation frequency around the limit cycle, where \( 1 - \delta \) reflects the deterministic rotational frequency controlled by the damping parameter \(\delta\). For \( n = 0 \), there’s no oscillation, only slow noise-driven decay; for \( n = \pm 1 \), it’s the fundamental frequency, and higher \( |n| \) gives faster oscillations, which are multiples (harmonics) of the basic rotation frequency $1 - \delta$.

For \( l \in \mathbb{Z}^+ \), it represents the radial mode number, often associated with the degree of Hermite polynomials in the eigenfunction expression (See Appendix~\ref{appendix:2d_sl_eqn}). Each \( l \) corresponds to a distinct radial mode with \( l \) nodes, describing perturbations away from the limit cycle radius. The real part \(-2l \delta\) reflects the decay rate of radial perturbations. Assuming \(\delta > 0\) (i.e., stable limit cycle case), this term \(-2l \delta\) is negative, and the decay rate increases linearly with \( l \), indicating that higher radial modes relax more quickly back to the limit cycle due to deterministic damping. The imaginary part \( i n (1 - \delta) \), which is identical to the \( l = 0 \) case, captures the oscillatory behavior tied to angular harmonics.

\noindent\textbf{Experiment Design and Result:} In our experiments, we aim to evaluate the accuracy of SDMD and EDMD in estimating eigenvalues of the Koopman generator of the stochastic Stuart-Landau system. The Fourier basis is selected as it aligns well with the periodic nature of the system. We specifically compare eigenvalues corresponding to the $l=0$ case in Eq.~\eqref{eq: sl_evalue} from the system with parameter settings: \( \delta = 0.25 \), \( \kappa = 1 \), \( \epsilon = 0.05 \), \( \gamma = 1 \), and \( \beta = 1 \). The system's state space is discretized over 20 points in both \( r \) and \( \theta \) which forms a uniform grid. The radii are uniformly sampled in the range \( r \in [0.4, 0.8] \), and the angles are uniformly sampled from \( \theta \in [-\pi, \pi] \). These points serve as initial conditions for the simulations. For the numerical integration, we let the internal integration step size of \( h = 1 \times 10^{-5} \) and \( n_\text{steps} = 10000 \) steps; so that the data is collected with $\Delta t = 0.1$. The true eigenvalues of the Koopman generator are computed for the case of \( l = 0 \). For this mode, we focus on eigenvalues corresponding to azimuthal harmonics, represented by \( n \in \{-10, -9, \dots, 10\} \) excluding \( n = 0 \), as this eigenvalue corresponds to a trivial mode. These analytical eigenvalues serve as the benchmark to evaluate the accuracy of SDMD and EDMD. In Figure~\ref{fig: 2d_sl_eqn_eval_comparison}, we show comparison of eigenvalues obtained from both methods. More discussion and comparison tests on eigenfunctions will be left in the Appendix \ref{appendix:2d_sl_eqn}.
\begin{figure}[htbp]
    \centering
    \begin{subfigure}[b]{0.48\textwidth}
        \includegraphics[width=\textwidth]{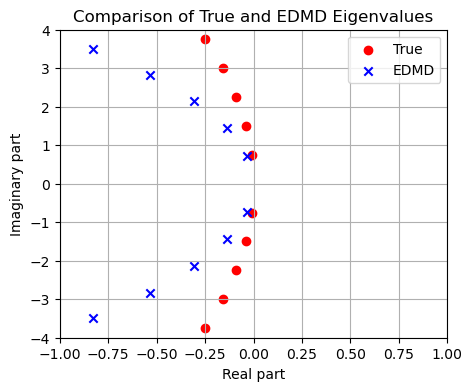}
        \caption{Eigenvalues by EDMD with Fourier basis}
        \label{fig: 2d_sl_eqn_eval_edmd}
    \end{subfigure}
    \hfill
    \begin{subfigure}[b]{0.48\textwidth}
        \includegraphics[width=\textwidth]{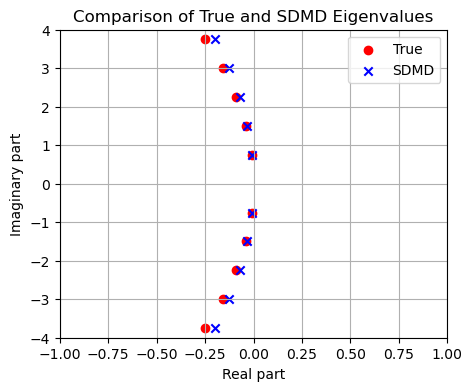}
        \caption{Eigenvalues by SDMD with Fourier basis}
        \label{fig: 2d_sl_eqn_eval_sdmd}
    \end{subfigure}
    \caption{Comparison of eigenvalues of Koopman generator estimated by EDMD and SDMD of the stochastic Stuart-Landau equation using Fourier basis. Figure \eqref{fig: 2d_sl_eqn_eval_edmd} shows the eigenvalues obtained from EDMD, while Figure \eqref{fig: 2d_sl_eqn_eval_sdmd} shows those obtained from SDMD.}
    \label{fig: 2d_sl_eqn_eval_comparison}
\end{figure}

%%%% Section 6.2 1D Ornstein–Uhlenbeck process %%%%
\subsection{1D Ornstein–Uhlenbeck Process}
The Ornstein–Uhlenbeck (OU) process is one of the few stochastic processes for which closed-form expressions can be derived for its transition probability density and the spectral properties of its Koopman generator. The one-dimensional OU process is described by the following SDE:
\begin{equation*}
    dX_t = \theta(\mu_0 - X_t)dt + \sigma dW_t,
\end{equation*}
where $\theta > 0$ is the mean reversion rate, $\mu_0$ is the long-term mean, $\sigma > 0$ is the volatility parameter, and $W_t$ is a standard Wiener process. The generator of this system acting on twice differentiable functions $f$ is \cite{oksendal2010stochastic,pavliotis2016stochastic}:
\begin{equation*}
    (\mathcal{A}f)(x) = \theta(\mu_0 - x)f'(x) + \frac{\sigma^2}{2}f''(x).
\end{equation*}

The eigenvalues and eigenfunctions of this generator have explicit forms, that is, for any non-negative integer $n$, the $n$-th eigenpair is:
\begin{equation}\label{ou_analytical_eigenpair}
    \lambda_n = -n\theta, \quad \phi_n(x) = H_n\left(\frac{x-\mu_0}{\sqrt{\sigma^2/(2\theta)}}\right),
\end{equation}
where $H_n(x)$ is the $n$-th order Hermite polynomial. The transition density solves the Fokker-Planck equation and has a Gaussian stationary distribution $\rho_\infty(x)$ with mean $\mu_0$ and variance $\sigma^2/(2\theta)$, specifically,
\begin{equation*}
    \rho_\infty(x) = \left(2\pi\sigma^2/(2\theta)\right)^{-1/2} \exp\left(-(x-\mu_0)^2/(\sigma^2/\theta)\right).
\end{equation*}
These analytical solutions provide rigorous benchmarks for testing numerical approximation methods like SDMD and gEDMD with neural network settings.

\noindent\textbf{Experiment Design:} 
The parameters are set as \( \theta = 1 \), \( \mu_0 = 0 \), and \( \sigma = 0.1 \). The EM method approximates the solution using discrete time steps \( h = 10^{-4} \), with the update rule:  
\[
    X_{t+h} = X_t + \theta (\mu_0 - X_t) h + \sigma \sqrt{h} \, \xi_t,
\]  
where \( \xi_t \sim \mathcal{N}(0, 1) \). In our setup, \( m = 10 \) initial points are chosen uniformly from the domain \([-2, 2]\), and for each initial point, the process is simulated over \( n_{\text{eval}} = 200 \) evaluations. We select sampling time interval $\Delta = 0.1$. These simulations generate the time series data required for learning the Koopman generator. To approximate the Koopman generator, we apply SDMD and gEDMD methods with a simple NN with dictionary size $N=20$. 

\noindent\textbf{Experiment Result:} 
Figure~\ref{fig: ou_process_eval_comparison}, we show the comparison of eigenpairs of Koopman generator $\mathcal{A}$ obtained from SDMD-DL, gEDMD-DL and EDMD-DL, respectively. Notice that our SDMD-DL method computes the approximated eigenvalues of the Koopman semigroup \( \mathcal{K}^{\Delta t} \), not of the Koopman generator $\mathcal{A}$, so the eigenvalues plotted are by Remark~\ref{rmk:generalized_eigen}. In Table~\ref{tab:sdmd_eigenvalues_ou}, we show the leading four approximated eigenvalues of the Koopman generator, which is first obtained by SDMD-DL and then computed by Eq.~\eqref{eq: evalue_generator_semigroup}. The Table \ref{tab:sdmd_eigenvalues_ou} exhibits two different test results and shows that SDMD-DL can successfully approximate the leading four eigenvalues of the generator since they are close to $0, -1, -2, -3$, as discussed in Eq.~\eqref{ou_analytical_eigenpair}. However, for eigenvalues corresponding to faster-decaying modes, the accuracy also diminishes. This limitation is likely due to insufficient data and an inadequate number of basis functions, constrained by the local computing resource. In Figure \ref{fig:ou_process_efunc_comparison}, we show that the eigenfunctions computed by SDMD-DL exhibit a polynomial structure in correct order, which also highlights the method's consistency with theoretical expectation as in Eq.~\eqref{eq: evalue_generator_semigroup}. To provide a clearer, baseline comparison of the core algorithms, we conducted another set of experiments under more idealized conditions. In this case, we used a fixed dictionary of monomial basis functions up to degree 5 and the analytical values for the SDE coefficients. The results are shown in Figure~\ref{fig: ou_process_monomial_comparison}. As expected, with a well-suited basis and perfect model knowledge, both SDMD and gEDMD achieve nearly perfect results, accurately capturing the theoretical spectrum. In contrast, standard EDMD still performs less accurate, as it does not correctly account for the stochastic term in its formulation.aware methods (SDMD and gEDMD) for stochastic systems.

%%%% Eigenvalue OU Process %%%%
\begin{figure*}[!htb]
    \centering    
    \begin{subfigure}[b]{0.33\linewidth}
        \centering
        \includegraphics[width=\linewidth]{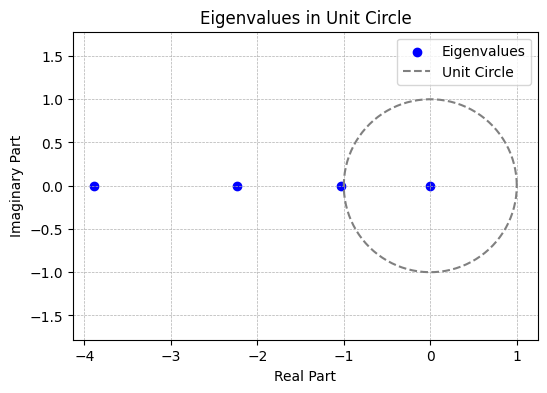}
        \caption{Test 1: SDMD-DL}
        \label{fig:subfig1}
    \end{subfigure}
    % \hfill
    \begin{subfigure}[b]{0.33\linewidth}
        \centering
        \includegraphics[width=\linewidth]{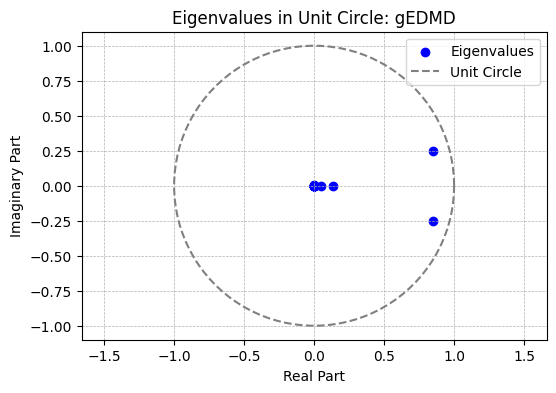}
        \caption{Test 1: gEDMD-DL}
        \label{fig:subfig2}
    \end{subfigure}
    % \hfill
    \begin{subfigure}[b]{0.33\linewidth}
        \centering
        \includegraphics[width=\linewidth]{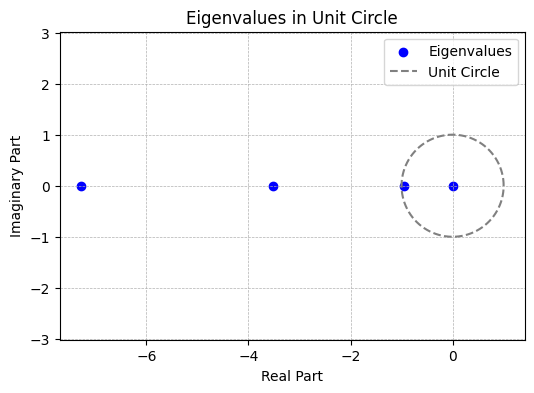}
        \caption{Test 1: EDMD-DL}
        \label{fig:subfig3}
    \end{subfigure}
    \begin{subfigure}[b]{0.33\linewidth}
        \centering
        \includegraphics[width=\linewidth]{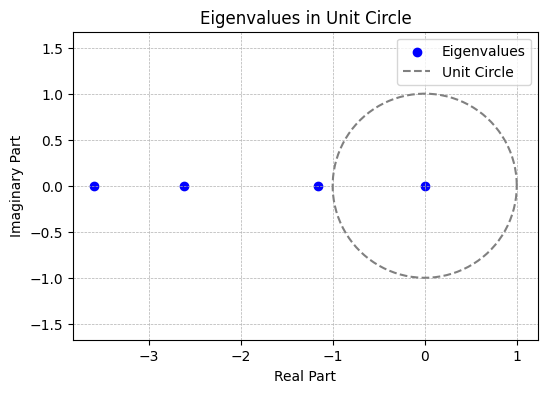}
        \caption{Test 2: SDMD-DL}
        \label{fig:subfig4}
    \end{subfigure}
    % \hfill
    \begin{subfigure}[b]{0.33\linewidth}
        \centering
        \includegraphics[width=\linewidth]{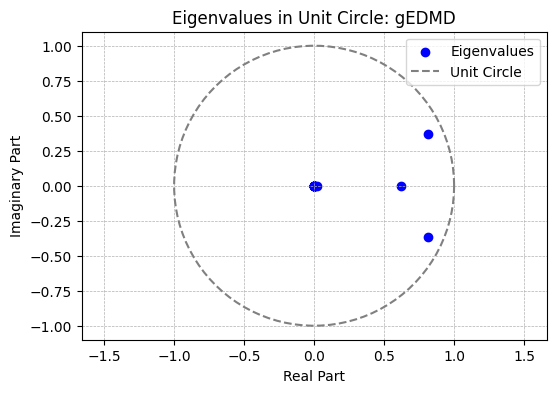}
        \caption{Test 2: gEDMD-DL}
        \label{fig:subfig5}
    \end{subfigure}
    % \hfill
    \begin{subfigure}[b]{0.33\linewidth}
        \centering
        \includegraphics[width=\linewidth]{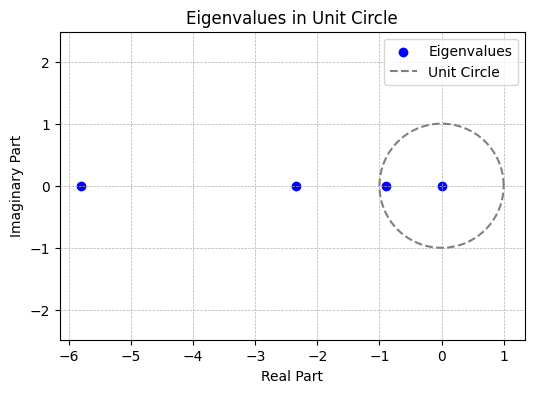}
        \caption{Test 2: EDMD-DL}
        \label{fig:subfig6}
    \end{subfigure}
    \caption{OU process: Two tests on computing eigenvalues of Koopman generator $\mathcal{A}$ obtained from SDMD-DL, gEDMD-DL and EDMD-DL.}
    \label{fig: ou_process_eval_comparison}
\end{figure*}

%%%% Eigenfunction OU Process %%%%
\begin{figure*}[!htb]
    \centering
    \begin{subfigure}[b]{0.33\linewidth}
        \centering
        \includegraphics[width=\linewidth]{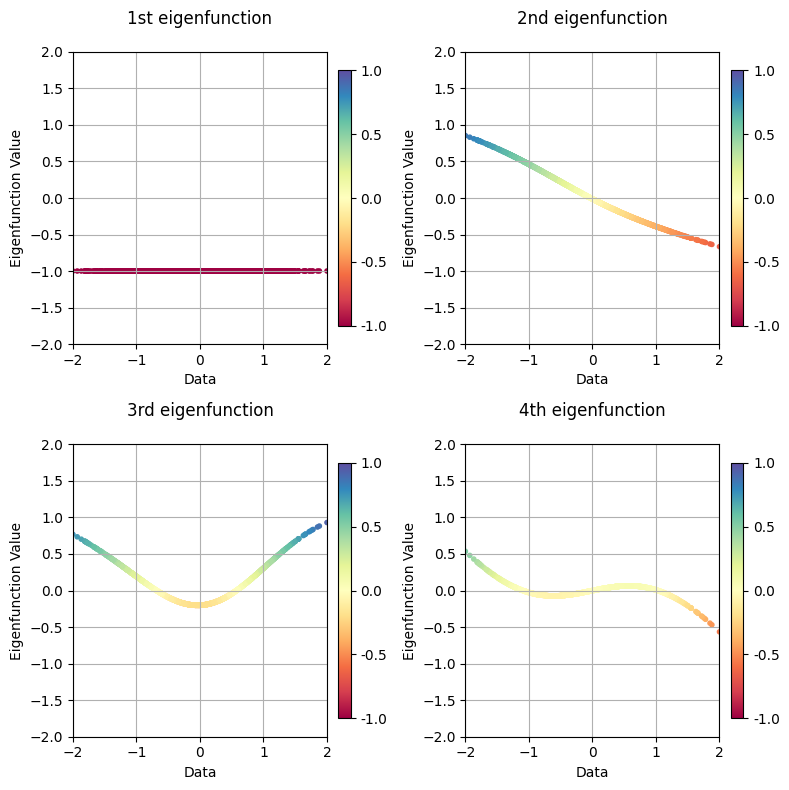}
        \caption{Test 1: SDMD-DL}
        \label{fig:subfig1}
    \end{subfigure}
    % \hfill
    \begin{subfigure}[b]{0.33\linewidth}
        \centering
        \includegraphics[width=\linewidth]{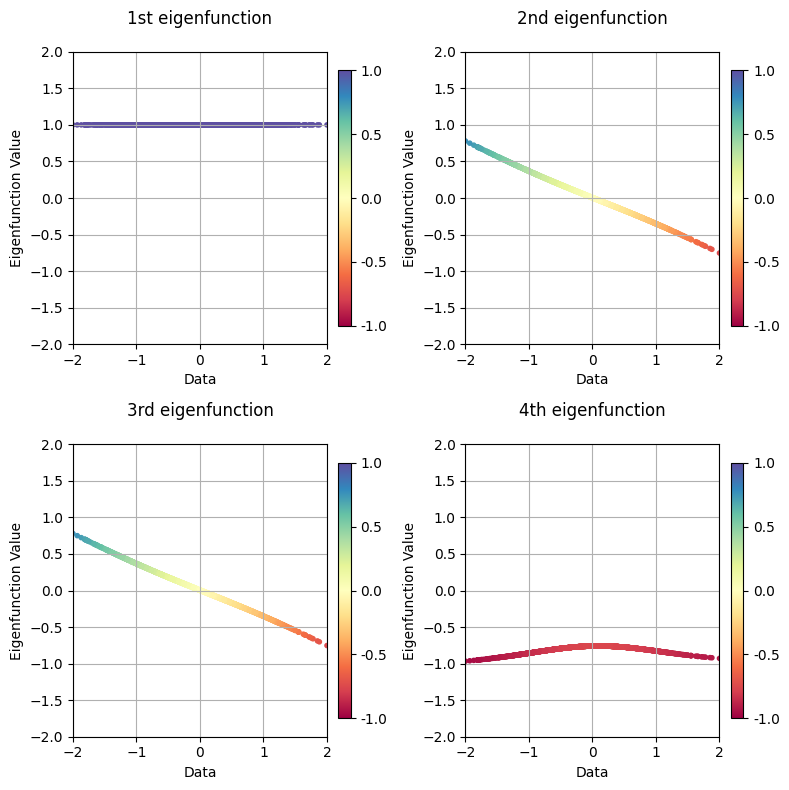}
        \caption{Test 1: gEDMD-DL}
        \label{fig:subfig2}
    \end{subfigure}
    % \hfill
    \begin{subfigure}[b]{0.33\linewidth}
        \centering
        \includegraphics[width=\linewidth]{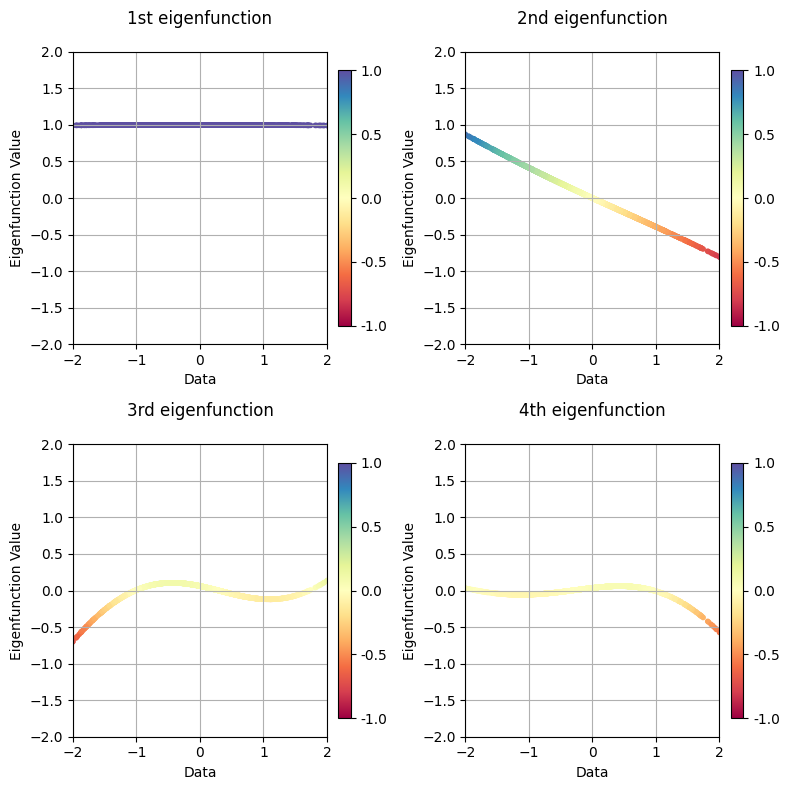}
        \caption{Test 1: EDMD-DL}
        \label{fig:subfig3}
    \end{subfigure}
    \begin{subfigure}[b]{0.33\linewidth}
        \centering
        \includegraphics[width=\linewidth]{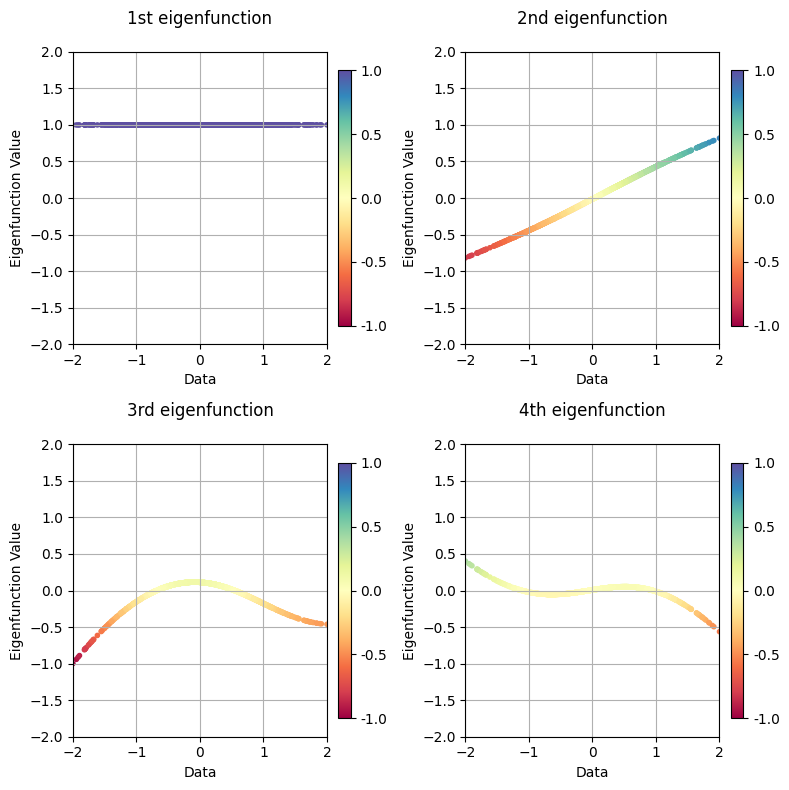}
        \caption{Test 2: SDMD-DL}
        \label{fig:subfig4}
    \end{subfigure}
    % \hfill
    \begin{subfigure}[b]{0.33\linewidth}
        \centering
        \includegraphics[width=\linewidth]{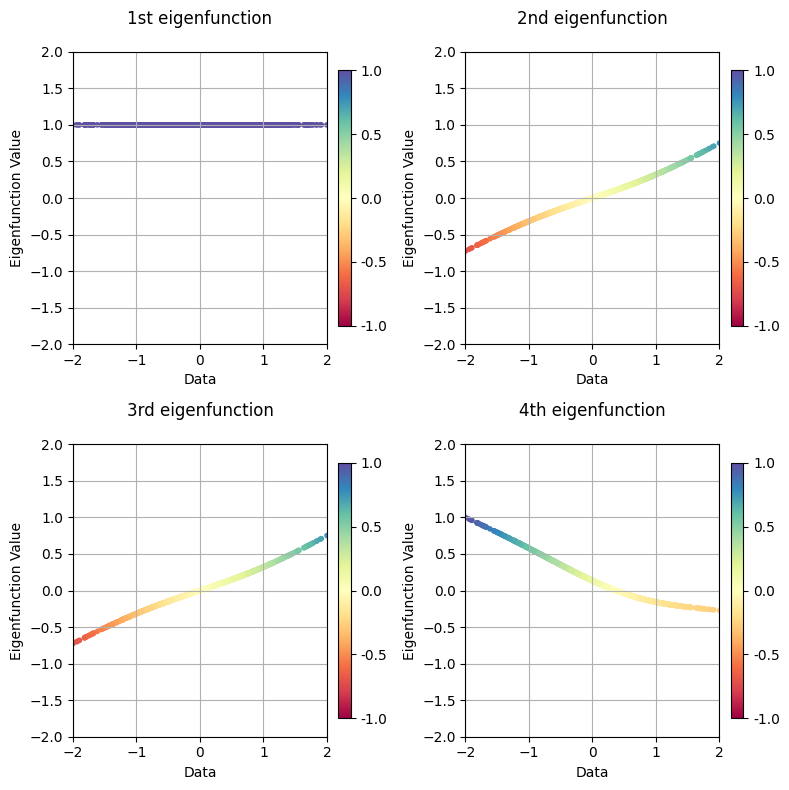}
        \caption{Test 2: gEDMD-DL}
        \label{fig:subfig5}
    \end{subfigure}
    % \hfill
    \begin{subfigure}[b]{0.33\linewidth}
        \centering
        \includegraphics[width=\linewidth]{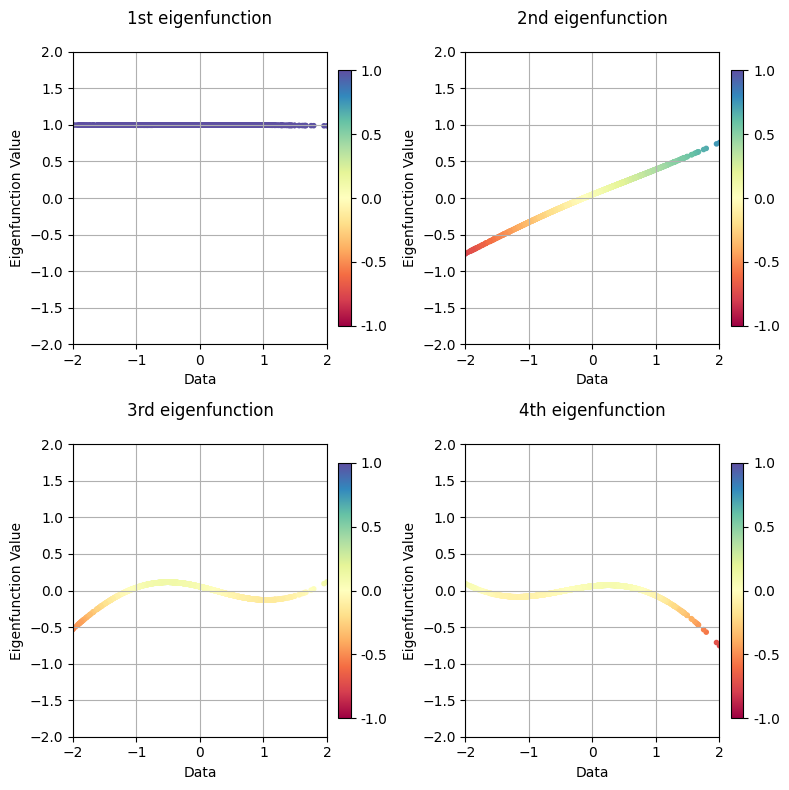}
        \caption{Test 2: EDMD-DL}
        \label{fig:subfig6}
    \end{subfigure}
    \caption{OU process: two tests on computing eigenfunctions obtained from SDMD-DL, gEDMD-DL and EDMD-DL.}
    \label{fig:ou_process_efunc_comparison}
\end{figure*}

%%%% Eigenvalue & Eigenfunction OU Process (Monomial) %%%%
\begin{figure*}[!htb]
    \centering    
    \begin{subfigure}[b]{0.33\linewidth}
        \centering
        \includegraphics[width=\linewidth]{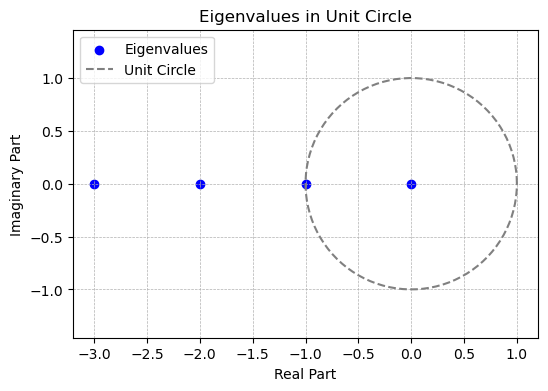}
        \caption{Test: SDMD}
        \label{fig:subfig1}
    \end{subfigure}
    % \hfill
    \begin{subfigure}[b]{0.33\linewidth}
        \centering
        \includegraphics[width=\linewidth]{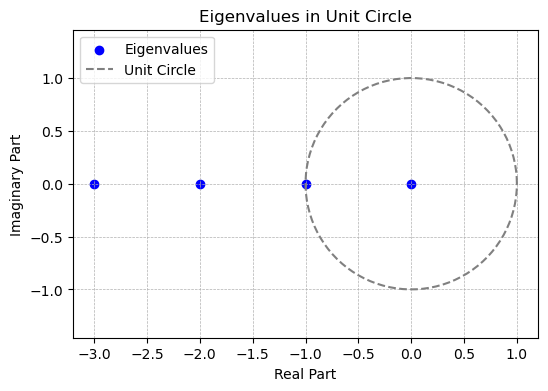}
        \caption{Test: gEDMD}
        \label{fig:subfig2}
    \end{subfigure}
    % \hfill
    \begin{subfigure}[b]{0.33\linewidth}
        \centering
        \includegraphics[width=\linewidth]{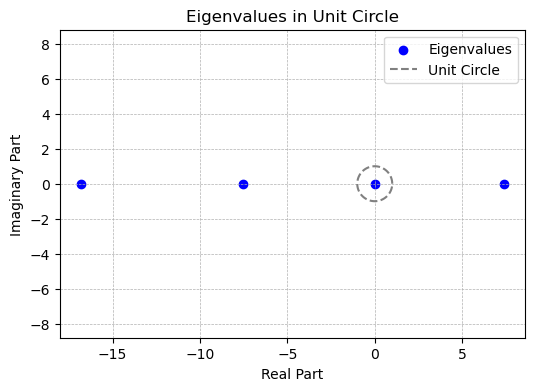}
        \caption{Test: EDMD}
        \label{fig:subfig3}
    \end{subfigure}
    \begin{subfigure}[b]{0.33\linewidth}
        \centering
        \includegraphics[width=\linewidth]{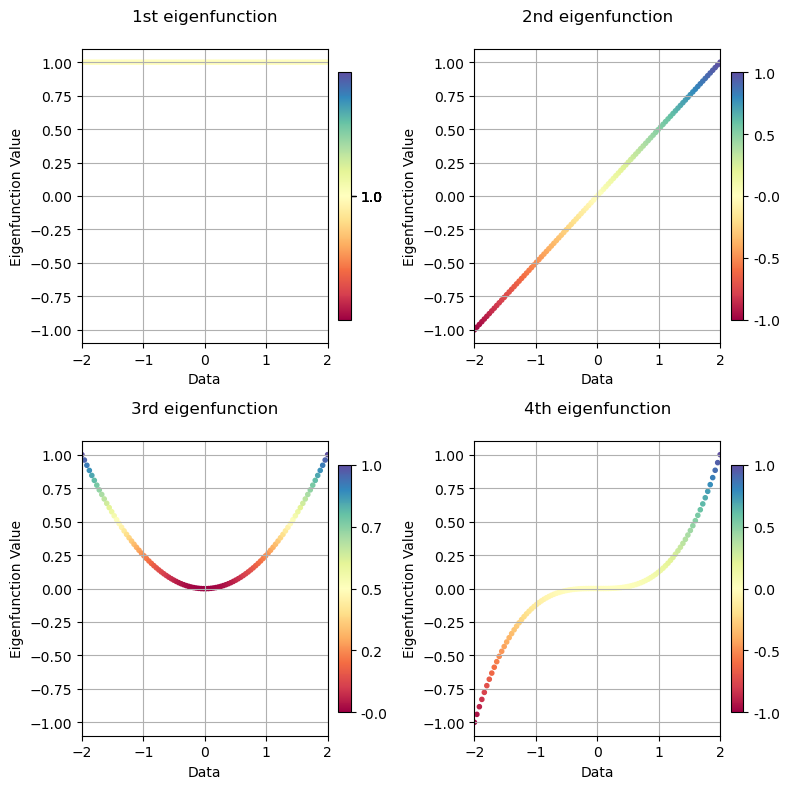}
        \caption{Test: SDMD}
        \label{fig:subfig1}
    \end{subfigure}
    % \hfill
    \begin{subfigure}[b]{0.33\linewidth}
        \centering
        \includegraphics[width=\linewidth]{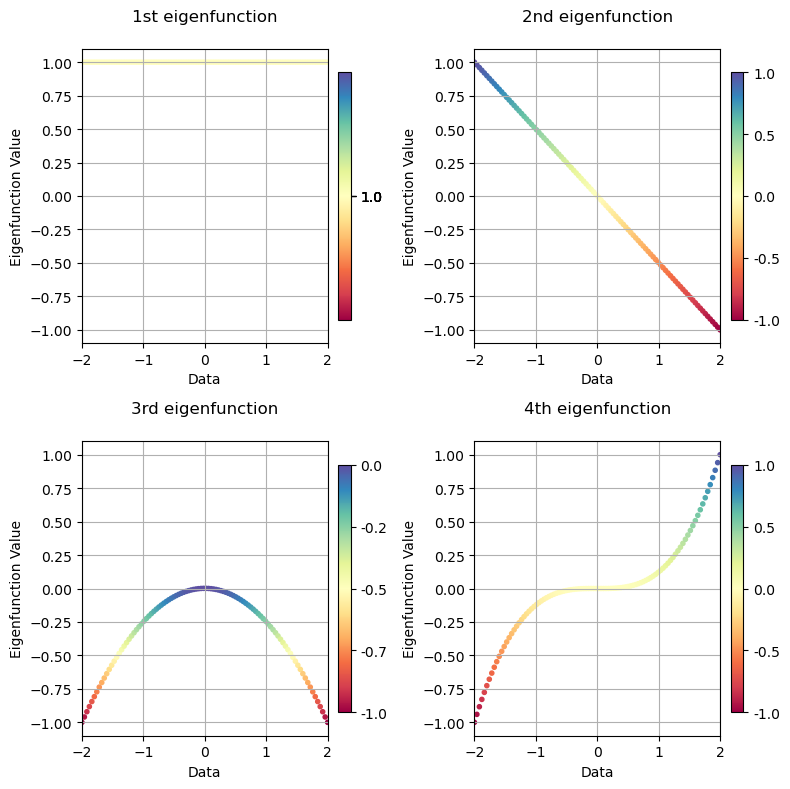}
        \caption{Test: gEDMD}
        \label{fig:subfig2}
    \end{subfigure}
    % \hfill
    \begin{subfigure}[b]{0.33\linewidth}
        \centering
        \includegraphics[width=\linewidth]{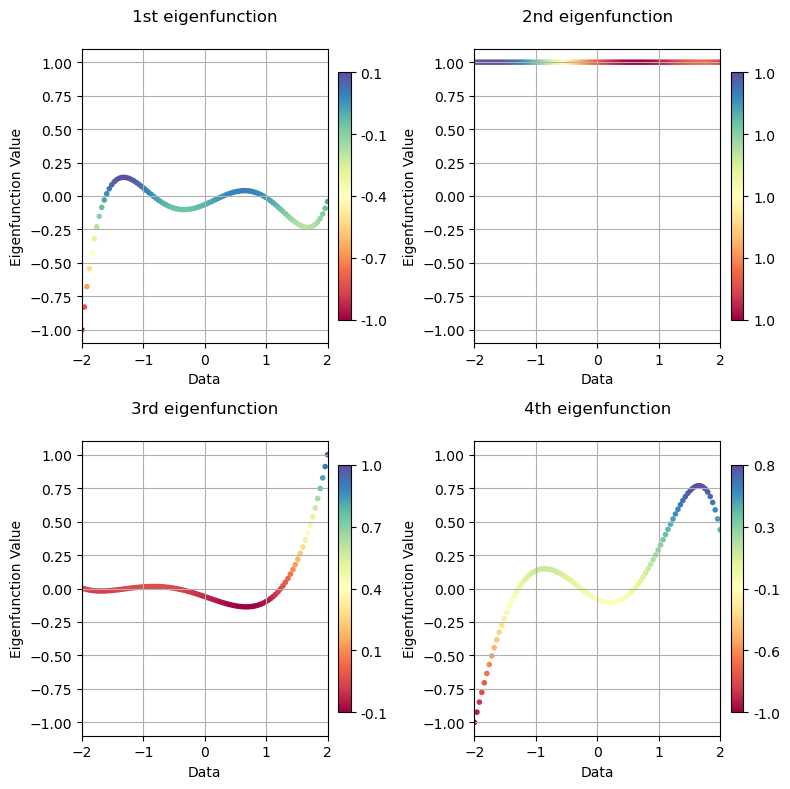}
        \caption{Test: EDMD}
        \label{fig:subfig3}
    \end{subfigure}
    \caption{OU process: test on computing eigenvalues and eigenfunctions of Koopman generator $\mathcal{A}$ obtained from SDMD, gEDMD and EDMD with monomial dictionary.}
    \label{fig: ou_process_monomial_comparison}
\end{figure*}

%%%% Table OU Process %%%%
\begin{table}[ht]
    \centering
    \begin{tabular}{|c|c|c|}
        \hline
        \textbf{Index} & \textbf{Test 1 Eigenvalues} & \textbf{Test 2 Eigenvalues} \\ \hline
        $\lambda_1$ & \(-5.75892737 \times 10^{-5}\) & \(-4.93801632 \times 10^{-5}\) \\ \hline
        $\lambda_2$ & \(-1.03753149\) & \(-1.15851652\) \\ \hline
        $\lambda_3$ & \(-2.23424145\) & \(-2.61595495\) \\ \hline
        $\lambda_4$ & \(-3.88774292\) & \(-3.59187336\) \\ \hline
    \end{tabular}
    \caption{OU process: First four eigenvalues of the Koopman generator estimated using SDMD.}    
    \label{tab:sdmd_eigenvalues_ou}
\end{table}

%%%% Section 6.3 2D Triple-Well system %%%%
\subsection{2D Triple-Well System}

The 2D triple-well potential system \cite{schutte2013metastability} provides a rich setting for studying stochastic dynamics due to its intricate energy landscape, which features three distinct basins of attraction separated by potential barriers. This system is inherently metastable, meaning that it tends to remain in one basin for long period, while random fluctuation will also cause slow transitions between basins. Meanwhile, within each basin the system experiences much quicker, small-scale fluctuation. Studying such metastable systems is valuable because it helps us understand long-term behavior, the occurrence of rare events, and the mechanism driving state changes. The dynamics of the system are governed by the following stochastic differential equation
\[
    d\mathbf{X}_t = -\nabla V(\mathbf{X}_t)dt + \bm{\sigma} d\mathbf{W}_t,
\]
where \( \mathbf{X}_t = (x_t, y_t) \) represents the system's state, \( V(x,y) \) is the potential function defining the energy landscape, \( \sigma \) is the diffusion coefficient that characterizes the intensity of the stochastic noise, and \( \mathbf{W}_t \) is a standard Wiener process.

\noindent\textbf{Experiment design:} In this experiment, the potential landscape $V(x, y)$ is defined by 
\[
    V(x, y) \coloneqq 
    3e^{-x^2 - (y - \frac{1}{3})^2}
    - 3e^{-x^2 - (y - \frac{5}{3})^2}
    - 5e^{-(x - 1)^2 - y^2}
    - 5e^{-(x + 1)^2 - y^2}
    + 0.2x^4
    + 0.2(y - \frac{1}{3})^4,
\]
which is depicted in Figure \ref{fig: triple_well_landscape}. For this system, the noise coefficient matrix $\bm{\sigma}$ is defined as a diagonal matrix, where each diagonal element is set to 1.09, that reflects the noise intensity in each spatial direction. The spatial domain is defined by \( x \in [-2, 2] \) and \( y \in [-1, 2] \), with \( m = 35 \) points uniformly sampled along each dimension to create a grid of $m^2$ initial conditions. The trajectories are generated using the Euler-Maruyama (EM) method for numerical integration, with an integration step size of \( h = 1 \times 10^{-3} \) and \( n_\text{steps} = 100 \) steps simulated for each trajectory. In this case, the collected data snapshots have time interval $\Delta t = 0.1$. We apply SDMD and gEDMD methods with dictionary size $N=10$. 
%%%% Potential Landscape %%%%
\begin{figure}[!htb]
    \centering    \includegraphics[width=0.35\linewidth]{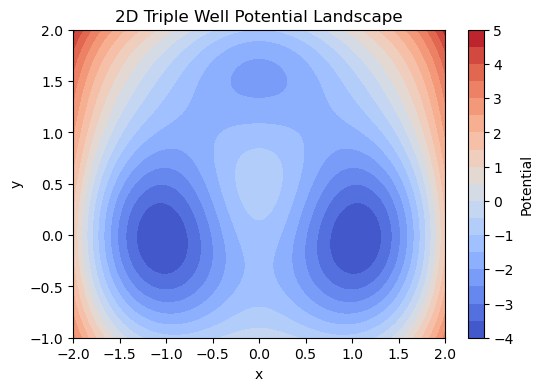}
    \caption{2D Triple Well Potential Landscape}  
    \label{fig: triple_well_landscape}  
\end{figure}

%%%% Eigenvalue Triple Well %%%%
\begin{figure}[!htb]
    \centering    
    \begin{subfigure}[b]{0.48\linewidth}
        \centering
        \includegraphics[width=\linewidth]{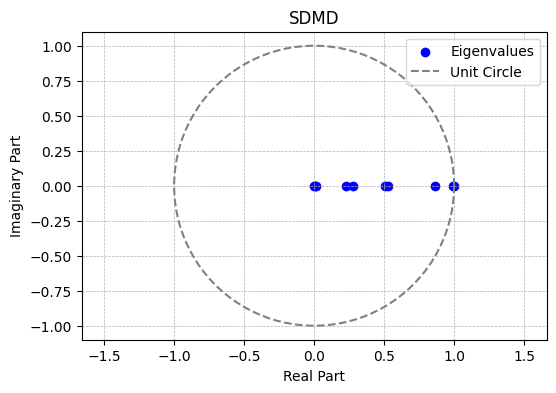}
        \caption{Test 1: SDMD-DL}
    \end{subfigure}
    % \hfill
    \begin{subfigure}[b]{0.48\linewidth}
        \centering
        \includegraphics[width=\linewidth]{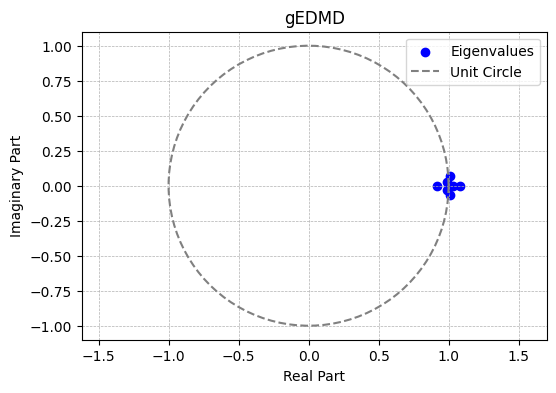}
        \caption{Test 1: gEDMD-DL}
    \end{subfigure}
    \begin{subfigure}[b]{0.48\linewidth}
        \centering
        \includegraphics[width=\linewidth]{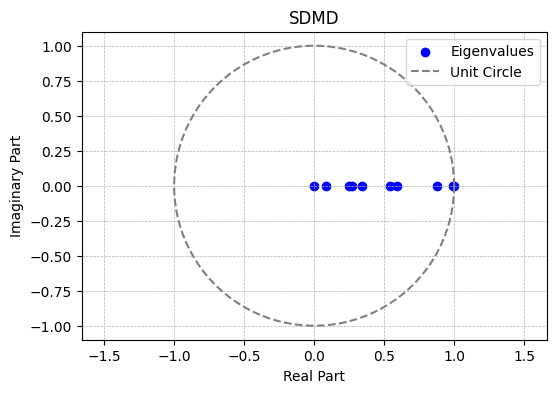}
        \caption{Test 2: SDMD-DL}
    \end{subfigure}
    % \hfill
    \begin{subfigure}[b]{0.48\linewidth}
        \centering
        \includegraphics[width=\linewidth]{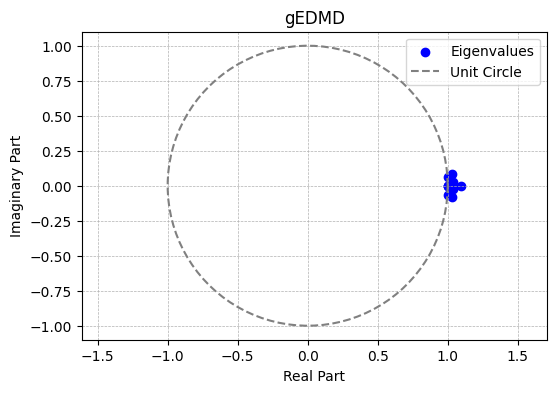}
        \caption{Test 2: gEDMD-DL}
    \end{subfigure}
    \caption{Triple-well: Two tests on computing eigenvalues of Koopman operator $\mathcal{K}^{\Delta t}$ obtained from SDMD-DL and gEDMD-DL.}
    \label{fig:triple_well_eval_comparison}
\end{figure}

%%%% Table Triple Well %%%%
\begin{table}[!htb]
    \centering
    \begin{tabular}{|c|c|c|}
        \hline
        \textbf{Index} & \textbf{Test 1 Eigenvalues} & \textbf{Test 2 Eigenvalues} \\ \hline
        $\lambda_1$ & \(0.999999135\) & \(1.00000083\) \\ \hline
        $\lambda_2$ & \(0.993172638\) & \(0.99359029\) \\ \hline
        $\lambda_3$ & \(0.892089359\) & \(0.86772991\) \\ \hline
    \end{tabular}
    \caption{Triple-well system: First three eigenvalues of the Koopman generator estimated using SDMD.}
    \label{tab:sdmd_eigenvalues_triple_well}
\end{table}

%%%% Eigenfunction Triple Well %%%%
\begin{figure}[!htb]
    \centering
    \begin{subfigure}[b]{0.98\linewidth}
        \centering
        \includegraphics[width=\linewidth]{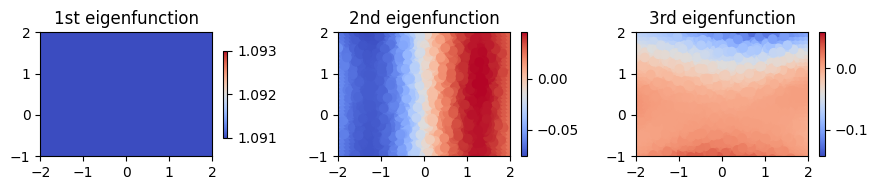}
        \caption{Test 1: SDMD-DL}
    \end{subfigure}
    % \hfill
    \begin{subfigure}[b]{0.98\linewidth}
        \centering
        \includegraphics[width=\linewidth]{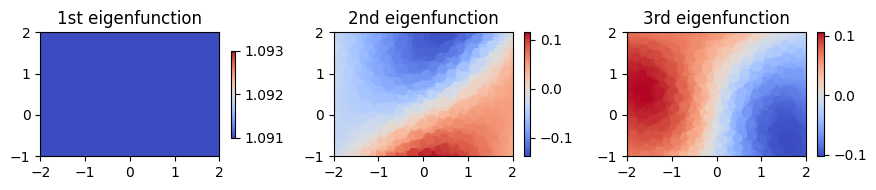}
        \caption{Test 1: gEDMD-DL}
    \end{subfigure}
    \begin{subfigure}[b]{0.98\linewidth}
        \centering
        \includegraphics[width=\linewidth]{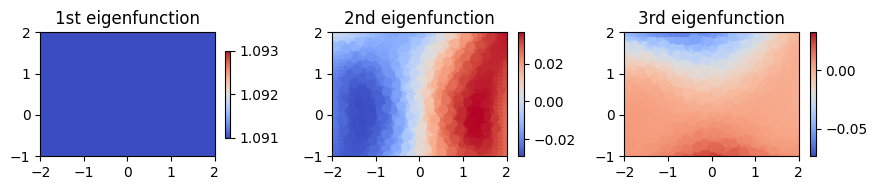}
        \caption{Test 2: SDMD-DL}
    \end{subfigure}
    % \hfill
    \begin{subfigure}[b]{0.98\linewidth}
        \centering
        \includegraphics[width=\linewidth]{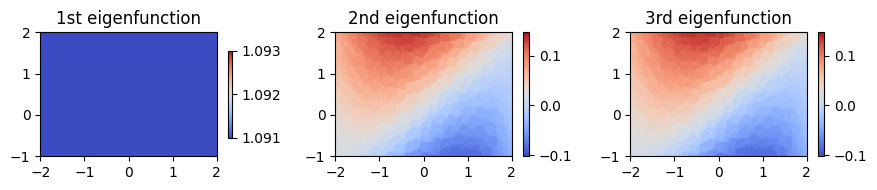}
        \caption{Test 2: gEDMD-DL}
    \end{subfigure}
    \caption{Triple-well: two tests on computing eigenfunctions obtained from SDMD-DL and gEDMD-DL.}
    \label{fig:triple_well_efunc_comparison}
\end{figure}

\noindent\textbf{Experiment Result:} Figure \ref{fig:triple_well_eval_comparison} displays the approximated eigenvalues of the Koopman semigroup $\mathcal{K}^{\Delta t}$ computed by SDMD. These eigenvalues quantify the rates of transitions between basins, with smaller eigenvalues corresponding to slower transitions that reflect the metastable behavior of the system. Here's a more detailed analysis for Figure~\ref{fig:triple_well_eval_comparison} and Figure~\ref{fig:triple_well_efunc_comparison}:
The eigenvalue \( \lambda = 1 \) corresponds to the system's steady state. The associated eigenfunction in this case is simply the constant function, which does not provide additional information about the system. It encodes features linked to the long-term behavior of the system. The eigenvalues close to 1 correspond to the slowest timescales in the system, indicating transitions between metastable states. In this triple well potential system, it represents transitions between the two deeper wells, as their similar depths create slow dynamics governed by the energy barrier between them. The slightly smaller eigenvalue reflects faster dynamics, particularly involving transitions with the shallower well. Since the shallower well is less stable, transitions involving this well occur more rapidly, resulting in a larger separation from 1. The smaller eigenvalues capture other transient behaviors that decay quickly. Table \ref{tab:sdmd_eigenvalues_triple_well} shows the eigenvalues of Koopman generator obtained by Eq.\eqref{eq: evalue_generator_semigroup}. These eigenvalues highlight the system's faster timescales and play a lesser role in describing long-term dynamics. This interpretation aligns with the expected behavior of metastable systems: the spectrum reflects both the number of basins and the hierarchy of transition rates among them. 

%%%% Section 6.4 Neural Mass Model %%%%
\subsection{Recovering Latent Timescales in a Neural Mass Model}

Our previous examples (Stuart-Landau oscillator, Ornstein-Uhlenbeck process, triple-well potential system) demonstrated SDMD’s accuracy in recovering known spectral structures. However, practical scenarios often involve latent dynamics, such as slow modulations or transient intermediate states that are not directly observable in measured variables. To illustrate SDMD's capacity to efficiently uncover these latent timescales with minimal observables and a compact dictionary, we present here a minimal neural-mass model \cite{montbrio2015macroscopic}. This example emphasizes SDMD’s effectiveness in extracting hidden multiscale structures from limited measurements, further underscoring the method's general applicability.

Specifically, we consider the following neural-mass model, a minimal two-dimensional reduction of a large population of quadratic integrate-and-fire neurons derived by Montbrio et. al. \cite{montbrio2015macroscopic}:
% \begin{align*}
% \dot r &= \frac{\Delta}{\pi} + 2\,r\,v,
% % \label{eq:neuralmass_r}
% \\
% \dot v &= v^2  + J\,r + I(t) - \pi^2\,r^2.
% % \label{eq:neuralmass_v}
% \end{align*}
\begin{align*}
    dr &= \left( \frac{\Delta}{\pi} + 2rv \right) dt + \sigma_r^2 \, dW_r(t), \\
    dv &= \left( v^2 + Jr + I(t) - \pi^2 r^2 \right) dt + \sigma_v^2 \, dW_v(t),
\end{align*}
where $\sigma_V$ and $\sigma_r$ are the standard deviation of addictive Gaussian noise.
This system intrinsically possesses two distinct dynamical regimes: a low-activity fixed-point regime and a damped-oscillatory (spiral) regime (see Figure~\ref{fig:neural_mass_model}A). The external input $I(t)$ can be set to two different values, each stabilizing one of these regimes (Figure~\ref{fig:neural_mass_model}B, top). Thus, the input acts as a latent slow-modulating state. Switching between input values induces a transient intermediate timescale associated with decay towards the spiral attractor. This provides a clear scenario to test the capability of SDMD to efficiently recover multiple latent timescales, including the slow modulation, transient decay, and fast oscillation from only the two observables $(r,v)$.

\begin{figure}[!htb]
  \centering
  \includegraphics[width=\textwidth]{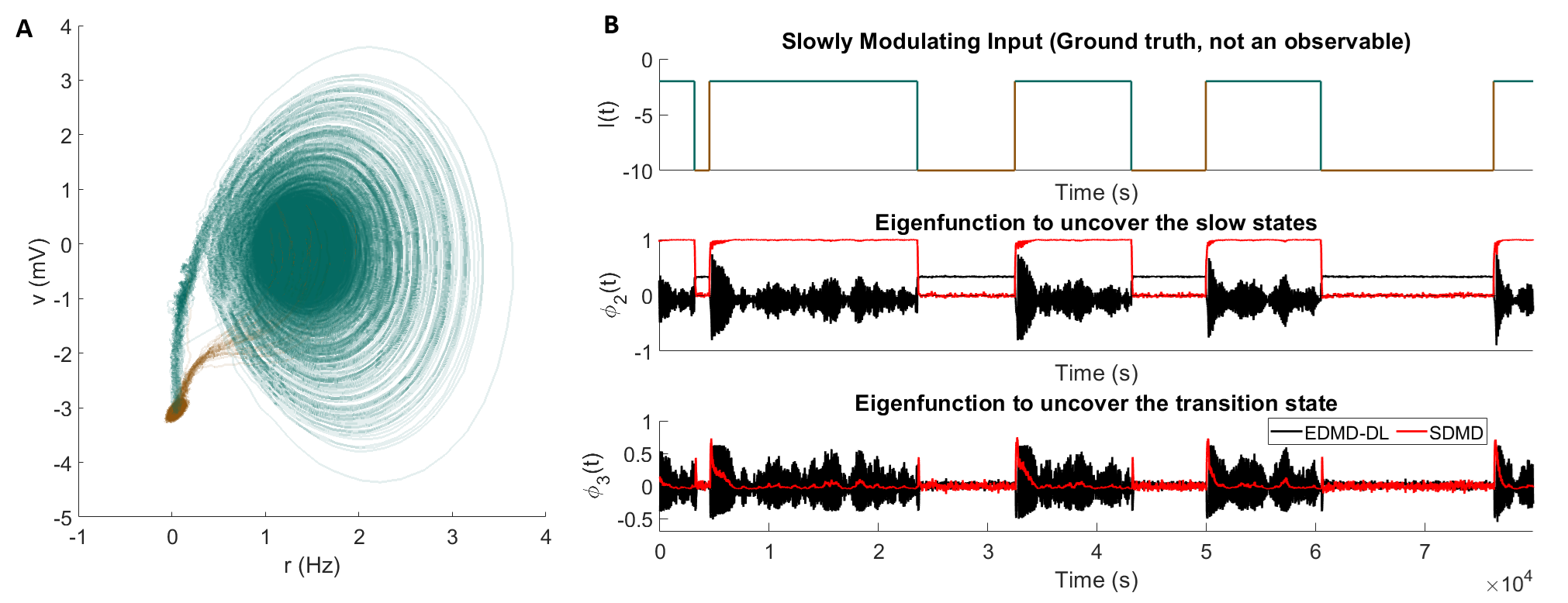}
  \caption{\textbf{SDMD recovers latent states and transient dynamics from a stochastic neural-mass model.}
  (\textbf{A}) Phase-space trajectory of the neural-mass model illustrating two distinct dynamical regimes: a low-activity fixed-point (brown) and a spiral oscillatory state (green), modulated by a slowly alternating latent input (not included as an observable or feature in either method). 
  (\textbf{B, top}) Ground truth of the slowly modulating input $I(t)$, switching between two values ($-2$ and $-10$), where the color code of the two states matches (A). This latent input state was not directly observable during analysis. 
  (\textbf{B, middle and bottom}) Dominant eigenfunctions estimated by SDMD (red) and EDMD-DL (black). SDMD clearly isolates the latent slow input modulation ($\phi_2$) and the transient dynamics during regime switches ($\phi_3$). By contrast, EDMD-DL fails to effectively recover either latent variable, with its eigenfunctions significantly influenced by high-frequency oscillations and stochastic noise.}
  \label{fig:neural_mass_model}
\end{figure}

\noindent\textbf{Experiment design:} We numerically integrated the above equations using the standard Euler method with a fixed time step ($\Delta t = 0.01$ s). The parameters were set as $\Delta = 1$, $J = 15$, $\sigma_V=\sigma_r=0.01$ and the input $I(t)$ alternated between two constant values ($I = -2$ and $I = -10$) driven by a Markov switching process. At each time step, the probability of maintaining the same input state was set to $0.9999$. Additionally, independent Gaussian noise with zero mean and standard deviation $0.01$ was added to both state variables. The total simulation duration was $6000$ s ($600,000$ data points), providing ample data to clearly identify latent slow modulations, intermediate transient dynamics, and fast oscillatory modes using SDMD. Both methods are trained with 300 outer epochs and 2 inner epochs, where the procedures are elaborated in Algorithm~\ref{algorithm: sdmd_dl} and in \cite{li2017extended}.

\noindent\textbf{Experiment results:} To examine the efficacy of SDMD relative to classical dictionary-learning approaches, we compared SDMD to Extended Dynamic Mode Decomposition with Dictionary Learning (EDMD-DL). Both methods used identical observable variables ($r$ and $v$) from the full $6000$ s simulated dataset. Crucially, no direct information about the latent input $I(t)$ was provided during the training process. To ensure a fair and consistent comparison, both SDMD and EDMD-DL employed dictionaries constructed via an identical three-layer neural network (50 neurons per hidden layer). The final dictionary consisted of $25$ neural-network-trained basis functions, supplemented by $2$ first-order polynomial observables ($r, v$) and a constant term, yielding a total of $28$ dictionary elements. This setup allowed for an objective evaluation of SDMD’s capability in recovering latent dynamical timescales and hidden states relative to standard EDMD-DL approaches.
It is worth mentioning that while EDMD-DL converges after 300 epochs, SDMD-DL does not converge. However, we measured the similarity between the normalized second Koopman eigenfunction (see following) and the ground truth latent input $I(t)$ with the Pearson correlation coefficient and picked the training result from the outer epoch with the highest similarity.

Figure~\ref{fig:neural_mass_model}B(middle, bottom) shows the first two dominant eigenfunctions identified by SDMD-DL and EDMD-DL from the neural-mass model data. SDMD successfully recovers a latent slow state (represented by eigenfunction $\phi_2$) that aligns well with the ground-truth partitions induced by the input modulation (Fig.~\ref{fig:neural_mass_model}B, top). Moreover, SDMD reveals another eigenfunction ($\phi_3$) clearly illustrating the transient dynamics during regime transitions. In contrast, EDMD-DL fails to capture these two latent states despite using the identical neural-network dictionary structure. Specifically, EDMD-DL’s eigenfunction $\phi_2$ does exhibit baseline shifts between different slow states, but it remains significantly corrupted by high-frequency oscillations from the spiral regime, failing to clearly isolate the slow modulation. Additionally, EDMD-DL's eigenfunction $\phi_3$ does not provide interpretable transient dynamics. 

Although we selected the SDMD epoch solely by maximizing the similarity between the normalized slow eigenfunction $\phi_2(t)$ and the ground-truth input $I(t)$, this choice also quite remarkably yields an accurate recovery of the transient switching dynamics in the third eigenfunction $\phi_3(t)$. In the selected epoch, $\phi_3(t)$ cleanly isolates the brief regime transitions, whereas in epochs with lower $\phi_2(t)-I(t)$ correlation it remains buried under high-frequency oscillations. This demonstrates that calibrating the Koopman operator on a single, dominant timescale can effectively tune the full spectral approximation, allowing both the slow modulations and the intermediate transients to emerge clearly.

In summary, this neural mass model example effectively demonstrates SDMD-DL’s advantage over EDMD-DL in extracting underlying multiscale dynamics from a stochastic system. Under identical conditions, i.e., limited observables, parametrized dictionaries, and neural-network architectures, SDMD accurately isolates latent slow and transient states, despite the presence of noise. The inherent robustness of SDMD to noise is highlighted by its capability to extract meaningful eigenfunctions even from noisy, partially observed data, whereas EDMD-DL fails to decouple latent structures from the stochastic fluctuation. These results illustrate SDMD’s suitability and superior performance for analyzing stochastic dynamical systems with hidden multiscale features. Moreover, our findings show that by selecting the model that best aligns the dominant slow eigenmode with the true latent input, one simultaneously calibrates the full Koopman operator, which enables clean recovery of both slow modulations and rapid switching dynamics.

%%%% Section 7 %%%%
\section{Conclusion}
In this paper, we presented a novel computational framework (SDMD) for estimating Koopman semigroups in stochastic dynamical systems. Our approach addresses several challenges in the field and provides a robust foundation for analyzing spectral properties in stochastic dynamics. By directly approximating the semigroup, it eliminates the need for expensive matrix exponential computations, significantly improving computational efficiency. These features, combined with a specially designed loss function and updating rule, make the framework particularly suitable for neural network implementations. Rigorous convergence analysis further supports the method’s reliability, which provides probabilistic error bounds and finite-dimensional approximations. Numerical experiments on the Stuart-Landau oscillator, Ornstein-Uhlenbeck process, triple-well potential, and neural mass model confirm that the framework not only estimates accurate Koopman eigenvalues and eigenfunctions but also recovers hidden slow-fast dynamics from limited noisy observations.

Future work will focus on three primary directions. First, we plan to extend the SDMD framework to high-dimensional systems. This extension will tackle both theoretical and computational challenges, including investigating how dimensionality influences convergence rates. Second, we aim to apply SDMD to real-world data to study multiscale phenomena, thereby bridging the gap between theory and practice in fields such as brain dynamics. Third, we intend to incorporate the resolvent operator into our analysis to derive explicit error bounds for the estimated Koopman semigroup, Koopman generator, and their spectral components. In addition, calibrating the approximation to emphasize a specific timescale by explicitly optimizing for a chosen eigenmode or weighting particular spectral components could offer a powerful new strategy for improving Koopman operator estimation and targeting distinct dynamical regimes. By analyzing the spectral characteristics of the resolvent operator, we hope to establish the theoretical framework of stability analysis for our SDMD method in stochastic dynamical systems.

%%%% Acknowledgement %%%%
\section*{Acknowledgement}
We would like to thank Prof. Igor Mezi\'{c} for bringing attention to the paper \cite{vcrnjaric2020koopman}, Prof. Jason Bramburger for pointing out some similar results and sharing valuable conclusion and example in Section~4.3\&6.5 of \cite{bramburger2024auxiliary} and Dr. Pietro Novelli for inspiring discussion of the experimental results and for reviewing the technical part during the early stages of this work.

%%%% References %%%%
\bibliography{references}

\begin{thebibliography}{10}

\bibitem{anderson1995computational}
John~David Anderson and John Wendt.
\newblock {\em Computational fluid dynamics}, volume 206.
\newblock Springer, 1995.

\bibitem{JMLR:v18:17-468}
Atilim~Gunes Baydin, Barak~A. Pearlmutter, Alexey~Andreyevich Radul, and Jeffrey~Mark Siskind.
\newblock Automatic differentiation in machine learning: a survey.
\newblock {\em Journal of Machine Learning Research}, 18(153):1--43, 2018.

\bibitem{benzi1982stochastic}
Roberto Benzi, Giorgio Parisi, Alfonso Sutera, and Angelo Vulpiani.
\newblock Stochastic resonance in climatic change.
\newblock {\em Tellus}, 34(1):10--16, 1982.

\bibitem{bolhuis2002transition}
Peter~G Bolhuis, David Chandler, Christoph Dellago, and Phillip~L Geissler.
\newblock Transition path sampling: Throwing ropes over rough mountain passes, in the dark.
\newblock {\em Annual review of physical chemistry}, 53(1):291--318, 2002.

\bibitem{boyd2013chebyshev}
J.P. Boyd.
\newblock {\em Chebyshev and Fourier Spectral Methods: Second Revised Edition}.
\newblock Dover Books on Mathematics. Dover Publications, 2013.

\bibitem{bramburger2024auxiliary}
Jason~J Bramburger and Giovanni Fantuzzi.
\newblock Auxiliary functions as {Koopman} observables: Data-driven analysis of dynamical systems via polynomial optimization.
\newblock {\em Journal of Nonlinear Science}, 34(1):8, 2024.

\bibitem{brunton2016discovering}
Steven~L Brunton, Joshua~L Proctor, and J~Nathan Kutz.
\newblock Discovering governing equations from data by sparse identification of nonlinear dynamical systems.
\newblock {\em Proceedings of the national academy of sciences}, 113(15):3932--3937, 2016.

\bibitem{chen2018neural}
Ricky~TQ Chen, Yulia Rubanova, Jesse Bettencourt, and David~K Duvenaud.
\newblock Neural ordinary differential equations.
\newblock {\em Advances in neural information processing systems}, 31, 2018.

\bibitem{Colbrook2023BeyondER}
Matthew~J. Colbrook, Qin Li, Ryan~V. Raut, and Alex Townsend.
\newblock Beyond expectations: Residual dynamic mode decomposition and variance for stochastic dynamical systems.
\newblock {\em ArXiv}, abs/2308.10697, 2023.

\bibitem{colbrook2024rigorous}
Matthew~J Colbrook and Alex Townsend.
\newblock Rigorous data-driven computation of spectral properties of {Koopman} operators for dynamical systems.
\newblock {\em Communications on Pure and Applied Mathematics}, 77(1):221--283, 2024.

\bibitem{vcrnjaric2020koopman}
Nelida {\v{C}}rnjari{\'c}-{\v{Z}}ic, Senka Ma{\'c}e{\v{s}}i{\'c}, and Igor Mezi{\'c}.
\newblock {Koopman} operator spectrum for random dynamical systems.
\newblock {\em Journal of Nonlinear Science}, 30:2007--2056, 2020.

\bibitem{deco2008dynamic}
Gustavo Deco, Viktor~K Jirsa, Peter~A Robinson, Michael Breakspear, and Karl Friston.
\newblock The dynamic brain: from spiking neurons to neural masses and cortical fields.
\newblock {\em PLoS computational biology}, 4(8):e1000092, 2008.

\bibitem{destexhe1999impact}
Alain Destexhe and Denis Par{\'e}.
\newblock Impact of network activity on the integrative properties of neocortical pyramidal neurons in vivo.
\newblock {\em Journal of neurophysiology}, 81(4):1531--1547, 1999.

\bibitem{1614066}
D.L. Donoho.
\newblock Compressed sensing.
\newblock {\em IEEE Transactions on Information Theory}, 52(4):1289--1306, 2006.

\bibitem{engel1999one}
K.J. Engel, S.~Brendle, R.~Nagel, M.~Campiti, T.~Hahn, G.~Metafune, G.~Nickel, D.~Pallara, C.~Perazzoli, A.~Rhandi, et~al.
\newblock {\em One-Parameter Semigroups for Linear Evolution Equations}.
\newblock Graduate Texts in Mathematics. Springer New York, 1999.

\bibitem{faisal2008noise}
A~Aldo Faisal, Luc~PJ Selen, and Daniel~M Wolpert.
\newblock Noise in the nervous system.
\newblock {\em Nature reviews neuroscience}, 9(4):292--303, 2008.

\bibitem{2021AGUFM.A15E1677G}
Dimitrios {Giannakis}, Gary {Froyland}, Benjamin {Lintner}, Max {Pike}, and Joanna {Slawinska}.
\newblock {Spectral analysis of climate dynamics with operator-theoretic approaches}.
\newblock In {\em AGU Fall Meeting Abstracts}, volume 2021, pages A15E--1677, December 2021.

\bibitem{gillespie1996exact}
Daniel~T Gillespie.
\newblock Exact numerical simulation of the ornstein-uhlenbeck process and its integral.
\newblock {\em Physical review E}, 54(2):2084, 1996.

\bibitem{gu2023stationary}
Yiqi Gu, John Harlim, Senwei Liang, and Haizhao Yang.
\newblock Stationary density estimation of it{\^o} diffusions using deep learning.
\newblock {\em SIAM Journal on Numerical Analysis}, 61(1):45--82, 2023.

\bibitem{hawkes2004stochastic}
A~Hawkes.
\newblock Stochastic modelling of single ion channels.
\newblock {\em Computational Neuroscience: a comprehensive approach}, pages 131--158, 2004.

\bibitem{ishikawa2024koopman}
Isao Ishikawa, Yuka Hashimoto, Masahiro Ikeda, and Yoshinobu Kawahara.
\newblock {Koopman} operators with intrinsic observables in rigged reproducing kernel {Hilbert} spaces.
\newblock {\em arXiv preprint arXiv:2403.02524}, 2024.

\bibitem{izhikevich2007dynamical}
Eugene~M Izhikevich.
\newblock {\em Dynamical systems in neuroscience}.
\newblock MIT press, 2007.

\bibitem{KLUS2020132416}
Stefan Klus, Feliks Nüske, Sebastian Peitz, Jan-Hendrik Niemann, Cecilia Clementi, and Christof Schütte.
\newblock Data-driven approximation of the {Koopman} generator: Model reduction, system identification, and control.
\newblock {\em Physica D: Nonlinear Phenomena}, 406:132416, 2020.

\bibitem{Koopman1931}
Bernard~O Koopman.
\newblock Hamiltonian systems and transformation in {Hilbert} space.
\newblock {\em Proceedings of the National Academy of Sciences}, 17(5):315, 1931.

\bibitem{koopman1932dynamical}
Bernard~O. Koopman and John~von Neumann.
\newblock Dynamical systems of continuous spectra.
\newblock {\em Proceedings of the National Academy of Sciences}, 18(3):255--263, 1932.

\bibitem{Korda_2017}
Milan Korda and Igor Mezić.
\newblock On convergence of extended dynamic mode decomposition to the {Koopman} operator.
\newblock {\em Journal of Nonlinear Science}, 28(2):687–710, November 2017.

\bibitem{kostic2022learningdynamicalsystemskoopman}
Vladimir Kostic, Pietro Novelli, Andreas Maurer, Carlo Ciliberto, Lorenzo Rosasco, and Massimiliano Pontil.
\newblock Learning dynamical systems via {Koopman} operator regression in reproducing kernel {Hilbert} spaces, 2022.

\bibitem{lasota2013chaos}
Andrzej Lasota and Michael~C Mackey.
\newblock {\em Chaos, fractals, and noise: stochastic aspects of dynamics}, volume~97.
\newblock Springer Science \& Business Media, 2013.

\bibitem{li2017extended}
Qianxiao Li, Felix Dietrich, Erik~M Bollt, and Ioannis~G Kevrekidis.
\newblock Extended dynamic mode decomposition with dictionary learning: A data-driven adaptive spectral decomposition of the {Koopman} operator.
\newblock {\em Chaos: An Interdisciplinary Journal of Nonlinear Science}, 27(10), 2017.

\bibitem{li2020fourier}
Zongyi Li, Nikola Kovachki, Kamyar Azizzadenesheli, Burigede Liu, Kaushik Bhattacharya, Andrew Stuart, and Anima Anandkumar.
\newblock Fourier neural operator for parametric partial differential equations.
\newblock {\em arXiv preprint arXiv:2010.08895}, 2020.

\bibitem{liu2024physics}
Yuying Liu, Aleksei Sholokhov, Hassan Mansour, and Saleh Nabi.
\newblock Physics-informed {Koopman} network for time-series prediction of dynamical systems.
\newblock In {\em ICLR 2024 Workshop on AI4DifferentialEquations In Science}, 2024.

\bibitem{llamazareselias2024datadrivenapproximationkoopmanoperators}
Liam Llamazares-Elias, Samir Llamazares-Elias, Jonas Latz, and Stefan Klus.
\newblock Data-driven approximation of {Koopman} operators and generators: Convergence rates and error bounds, 2024.

\bibitem{lu2019deeponet}
Lu~Lu, Pengzhan Jin, and George~Em Karniadakis.
\newblock Deeponet: Learning nonlinear operators for identifying differential equations based on the universal approximation theorem of operators.
\newblock {\em arXiv preprint arXiv:1910.03193}, 2019.

\bibitem{lusch2018deep}
Bethany Lusch, J~Nathan Kutz, and Steven~L Brunton.
\newblock Deep learning for universal linear embeddings of nonlinear dynamics.
\newblock {\em Nature communications}, 9(1):4950, 2018.

\bibitem{Mann01112016}
Jordan Mann and J.~Nathan Kutz.
\newblock Dynamic mode decomposition for financial trading strategies.
\newblock {\em Quantitative Finance}, 16(11):1643--1655, 2016.

\bibitem{manwani1999detecting}
Amit Manwani and Christof Koch.
\newblock Detecting and estimating signals in noisy cable structures, i: Neuronal noise sources.
\newblock {\em Neural computation}, 11(8):1797--1829, 1999.

\bibitem{math9192495}
Alexandre Mauroy.
\newblock {Koopman} operator framework for spectral analysis and identification of infinite-dimensional systems.
\newblock {\em Mathematics}, 9(19), 2021.

\bibitem{mezic2005spectral}
Igor Mezi{\'c}.
\newblock Spectral properties of dynamical systems, model reduction and decompositions.
\newblock {\em Nonlinear Dynamics}, 41:309--325, 2005.

\bibitem{mezic2013analysis}
Igor Mezi{\'c}.
\newblock Analysis of fluid flows via spectral properties of the {Koopman} operator.
\newblock {\em Annual review of fluid mechanics}, 45(1):357--378, 2013.

\bibitem{montbrio2015macroscopic}
Ernest Montbri{\'o}, Diego Paz{\'o}, and Alex Roxin.
\newblock Macroscopic description for networks of spiking neurons.
\newblock {\em Physical Review X}, 5(2):021028, 2015.

\bibitem{noe2013variational}
Frank No{\'e} and Feliks Nuske.
\newblock A variational approach to modeling slow processes in stochastic dynamical systems.
\newblock {\em Multiscale Modeling \& Simulation}, 11(2):635--655, 2013.

\bibitem{oksendal2010stochastic}
B.~{\O}ksendal.
\newblock {\em Stochastic Differential Equations: An Introduction with Applications}.
\newblock Universitext. Springer Berlin Heidelberg, 2010.

\bibitem{pavliotis2016stochastic}
G.A. Pavliotis.
\newblock {\em Stochastic Processes and Applications: Diffusion Processes, the Fokker-Planck and Langevin Equations}.
\newblock Texts in Applied Mathematics. Springer New York, 2016.

\bibitem{pazy2012semigroups}
A.~Pazy.
\newblock {\em Semigroups of Linear Operators and Applications to Partial Differential Equations}.
\newblock Applied Mathematical Sciences. Springer New York, 2012.

\bibitem{philipp2024variancerepresentationsconvergencerates}
Friedrich~M. Philipp, Manuel Schaller, Septimus Boshoff, Sebastian Peitz, Feliks Nüske, and Karl Worthmann.
\newblock Variance representations and convergence rates for data-driven approximations of {Koopman} operators, 2024.

\bibitem{rowley2009spectral}
Clarence~W Rowley, Igor Mezi{\'c}, Shervin Bagheri, Philipp Schlatter, and Dan~S Henningson.
\newblock Spectral analysis of nonlinear flows.
\newblock {\em Journal of fluid mechanics}, 641:115--127, 2009.

\bibitem{rudy2017data}
Samuel~H Rudy, Steven~L Brunton, Joshua~L Proctor, and J~Nathan Kutz.
\newblock Data-driven discovery of partial differential equations.
\newblock {\em Science advances}, 3(4):e1602614, 2017.

\bibitem{ref1}
Claude Sammut and Geoffrey~I. Webb, editors.
\newblock {\em McDiarmid's Inequality}, pages 651--652.
\newblock Springer US, Boston, MA, 2010.

\bibitem{schutte2023overcoming}
Christof Sch{\"u}tte, Stefan Klus, and Carsten Hartmann.
\newblock Overcoming the timescale barrier in molecular dynamics: Transfer operators, variational principles and machine learning.
\newblock {\em Acta Numerica}, 32:517--673, 2023.

\bibitem{schutte2013metastability}
Christof Sch{\"u}tte and Marco Sarich.
\newblock {\em Metastability and Markov state models in molecular dynamics}, volume~24.
\newblock American Mathematical Soc., 2013.

\bibitem{strogatz2024nonlinear}
Steven~H Strogatz.
\newblock {\em Nonlinear dynamics and chaos: with applications to physics, biology, chemistry, and engineering}.
\newblock Chapman and Hall/CRC, 2024.

\bibitem{tantet2020ruellepollicottresonancesstochasticsystems}
Alexis Tantet, Mickaël~D. Chekroun, Henk~A. Dijkstra, and J.~David Neelin.
\newblock Ruelle-pollicott resonances of stochastic systems in reduced state space. part ii: Stochastic hopf bifurcation, 2020.

\bibitem{Tu2014}
Jonathan~H. Tu, Clarence~W. Rowley, Dirk~M. Luchtenburg, Steven~L. Brunton, and J.~Nathan Kutz.
\newblock On dynamic mode decomposition: Theory and applications.
\newblock {\em Journal of Computational Dynamics}, 1(2):391--421, 2014.

\bibitem{tuckell1988introduction}
HC~Tuckell.
\newblock Introduction to theoretical neurobiology: Volume 2, nonlinear and stochastic theories, 1988.

\bibitem{Vershynin_2018}
Roman Vershynin.
\newblock {\em High-Dimensional Probability: An Introduction with Applications in Data Science}.
\newblock Cambridge Series in Statistical and Probabilistic Mathematics. Cambridge University Press, 2018.

\bibitem{wanner2022robust}
Mathias Wanner and Igor Mezic.
\newblock Robust approximation of the stochastic {Koopman} operator.
\newblock {\em SIAM Journal on Applied Dynamical Systems}, 21(3):1930--1951, 2022.

\bibitem{Williams_2015}
Matthew~O. Williams, Ioannis~G. Kevrekidis, and Clarence~W. Rowley.
\newblock A data–driven approximation of the {Koopman} operator: Extending dynamic mode decomposition.
\newblock {\em Journal of Nonlinear Science}, 25(6):1307–1346, June 2015.

\bibitem{10.1063/5.0157763}
Wei Xiong, Muyuan Ma, Xiaomeng Huang, Ziyang Zhang, Pei Sun, and Yang Tian.
\newblock {KoopmanLab}: Machine learning for solving complex physics equations.
\newblock {\em APL Machine Learning}, 1(3):036110, 09 2023.

\bibitem{xu2025reskoopnetlearningkoopmanrepresentations}
Yuanchao Xu, Kaidi Shao, Nikos Logothetis, and Zhongwei Shen.
\newblock Reskoopnet: Learning {Koopman} representations for complex dynamics with spectral residuals, 2025.

\end{thebibliography}
\bibliographystyle{plain}

\newpage
\appendix

%%%% Section A %%%%
% \section{Appendix}

\section{Second-order Stochastic Taylor Expansion in SDMD}\label{2nd_order_expansion}
For $f\in C^4_b(\mathcal{M})$, the Koopman semigroup admits the 2nd order stochastic Taylor expansion in the following
$$
    \mathcal{K}^{\Delta t} f = f + \Delta t\,\mathcal{A}f + \frac{\Delta t^2}{2}\,\mathcal{A}^2 f + o(\Delta t^2),
$$
where the generator $\mathcal{A}$ acts as
$$
    \mathcal{A}f = \sum_{i=1}^d b_i \frac{\partial f}{\partial x_i} + \frac12 \sum_{i,j=1}^d (\sigma\sigma^\top)_{ij} \frac{\partial^2 f}{\partial x_i \partial x_j}.
$$
Applying $\mathcal{A}$ again yields
$$
\mathcal{A}^2 f = \sum_{k=1}^d b_k \frac{\partial}{\partial x_k} \left( \mathcal{A}f \right) + \frac12 \sum_{k,\ell=1}^d (\sigma\sigma^\top)_{k\ell} \frac{\partial^2}{\partial x_k \partial x_\ell} \left( \mathcal{A}f \right),
$$
which expands into a combination of terms involving derivatives of $b$ and $\sigma$ up to second order, and derivatives of $f$ up to fourth order. Specifically,
$$
\mathcal{A}^2 f = \sum_{i,k} b_k \frac{\partial b_i}{\partial x_k} \frac{\partial f}{\partial x_i} + \cdots + \frac14 \sum_{i,j,k,\ell} (\sigma\sigma^\top)_{k\ell} \frac{\partial^2 (\sigma\sigma^\top)_{ij}}{\partial x_k \partial x_\ell} \frac{\partial^2 f}{\partial x_i \partial x_j} + \cdots,
$$
where the omitted terms include all mixed derivatives up to order four of $f$ multiplied by appropriate derivatives of $b$ and $\sigma$. 
In the SDMD context, incorporating this term for each basis function $\psi_j$ would modify Eq.~\eqref{eq: expansion_basis} to
$$
\mathcal{K}^{\Delta t} \psi_j(x) \approx \psi_j(x) + \Delta t\,(\mathcal{A}\psi_j)(x) + \frac{\Delta t^2}{2}\,(\mathcal{A}^2\psi_j)(x),
$$
which leads to a higher-order approximation but with significantly greater computational cost due to the evaluation of $\mathcal{A}^2\psi_j$.

\section{McDiarmid's Inequality}\label{McDiarmid's Inequality}
\begin{definition}[Bounded Differences Property]\label{def: bounded_diff_property}
    A function $f : \mathcal{X}_1 \times \mathcal{X}_2 \times \cdots \times \mathcal{X}_n \to \mathbb{R}$ satisfies the \textit{bounded differences property} if substituting the value of the $i$-th coordinate $x_i$ changes the value of $f$ by at most $c_i$. More formally, if there are constants $c_1, c_2, \ldots, c_n$ such that for all $i \in [n]$, and all $x_1 \in \mathcal{X}_1, x_2 \in \mathcal{X}_2, \ldots, x_n \in \mathcal{X}_n$,
    \[
        \sup_{x_i' \in \mathcal{X}_i} |f(x_1, \ldots, x_{i-1}, x_i, x_{i+1}, \ldots, x_n) - f(x_1, \ldots, x_{i-1}, x_i', x_{i+1}, \ldots, x_n)| \leq c_i.
    \]
\end{definition}

\begin{lemma}[McDiarmid's Inequality \cite{ref1}]\label{lemma: McDiarmid_ineq}
    Let $f : \mathcal{X}_1 \times \mathcal{X}_2 \times \cdots \times \mathcal{X}_n \to \mathbb{R}$ satisfy the bounded differences property with bounds $c_1, c_2, \ldots, c_n$ as in Definition \ref{def: bounded_diff_property}. Consider independent random variables $X_1, X_2, \ldots, X_n$ where $X_i \in \mathcal{X}_i$ for all $i$. Then, for any $\epsilon > 0$,
    \[
        \mathbb{P}(|f(X_1, X_2, \ldots, X_n) - \mathbb{E}[f(X_1, X_2, \ldots, X_n)]| \geq \epsilon) \leq 2 \exp\left( - \frac{2\epsilon^2}{\sum_{i=1}^n c_i^2} \right).
    \]
\end{lemma}

% \subsection{First Trotter--Kato Approximation Theorem}
% \begin{theorem}[First Trotter--Kato Approximation Theorem]\label{TKAT}
%     Let $(\mathcal{K}^t)_{t\geq0}$ and $(\mathcal{K}_N^t)_{t\geq0}$, $N \in \mathbb{N}$, be strongly continuous semigroups on $M$ with generators $\mathcal{A}$ and $\mathcal{A}_N$, respectively, and assume that they satisfy the estimate
%     \[
%     \|\mathcal{K}^t\|, \|\mathcal{K}_N^t\| \leq De^{\omega t} \quad \text{for all } t \geq 0, N \in \mathbb{N},
%     \]
%     and some constants $D \geq 1$, $\omega \in \mathbb{R}$. Take $\mathcal{D}$ to be a core for $\mathcal{A}$ and consider the following assertions.
%     \begin{itemize}
%         \item[(a)] $\mathcal{D} \subset \mathcal{D}(\mathcal{A}_N)$ for all $N \in \mathbb{N}$ and $\mathcal{A}_N f \to \mathcal{A} f$ for all $f \in \mathcal{D}$.
%         \item[(b)] $\mathcal{K}_N^t f \to \mathcal{K}^t f$ for all $f \in \mathcal{F}$, uniformly for $t$ in compact interval.
%     \end{itemize}
%     Then the implication $\text{(a)} \implies \text{(b)}$ holds, while $(b)$ does not imply $(a)$.
% \end{theorem}
% \begin{proof}
%     See \cite[Section 4.8]{engel1999one}
% \end{proof}

\section{2D Stuart-Landau equation: Phase Diffusion Equation}\label{appendix:2d_sl_eqn}

The SL equation not only has the Cartesian and polar coordinates form, but also has phase coordinates form \cite[Section 4]{tantet2020ruellepollicottresonancesstochasticsystems} given in the following:
\[
\begin{aligned}
    dr &= (\delta r - \kappa r^3 + \frac{\epsilon^2}{2r}) \, dt + \epsilon \, dW_r, \\
    d\phi &= \omega_f \, dt + \frac{\epsilon}{r} \, dW_\theta - \tilde{\beta} \frac{\epsilon}{r} \, dW_r,
\end{aligned}
\]
where \( \phi = \theta - \tilde{\beta}\log(r/R) \) with \( \tilde{\beta} = \beta/\kappa \) being the twist factor.

For $\delta > 0$ and $\epsilon\sqrt{\kappa}/\delta \ll 1$, the eigenfunctions are given by:
\begin{itemize}
    \item $l = 0$:
    \[
        \psi_{0n} = \exp(i(n(\theta - \tilde{\beta}\log\frac{r}{R}))),
    \]
    \item $l>0$:
    \[
        \psi_{ln} \propto H_l(\sqrt{2\delta}\frac{r-R}{\epsilon})\exp(i(n(\theta - \tilde{\beta}\log\frac{r}{R}))),
    \]
\end{itemize}
where $H_l$ is the $l$-th order Hermite polynomial.

Figure \ref{fig: 2d_sl_eqn_efun_comparison} displays a comparison of the analytical eigenfunctions with those obtained from EDMD and SDMD for various modes. The analytical eigenfunctions are normalized by a factor of \(1/\sqrt{2^{|n|}|n|!}\) and combine the phase dynamics, expressed by \(\exp\left(in\left(\theta - \tilde{\beta}\log\left(\frac{r}{R}\right)\right)\right)\), with the radial structure given by Hermite polynomials \(H_l\left(\frac{\sqrt{2\delta}(r-R)}{\epsilon}\right)\). Both EDMD and SDMD utilize Fourier basis functions; however, SDMD demonstrates superior accuracy in capturing the phase structure. Note that our dataset covers the range \(r \in [0.4, 0.8]\), so there is no information available for the region \(r < 0.4\). Within the considered range, the eigenfunctions computed by both methods exhibit the rotational behavior observed in the true eigenfunctions, as shown in Figure \ref{fig: 2d_sl_eqn_efun_comparison}.
\begin{figure}[H]
    \centering
    \begin{subfigure}[b]{0.33\textwidth}
        \includegraphics[width=\textwidth]{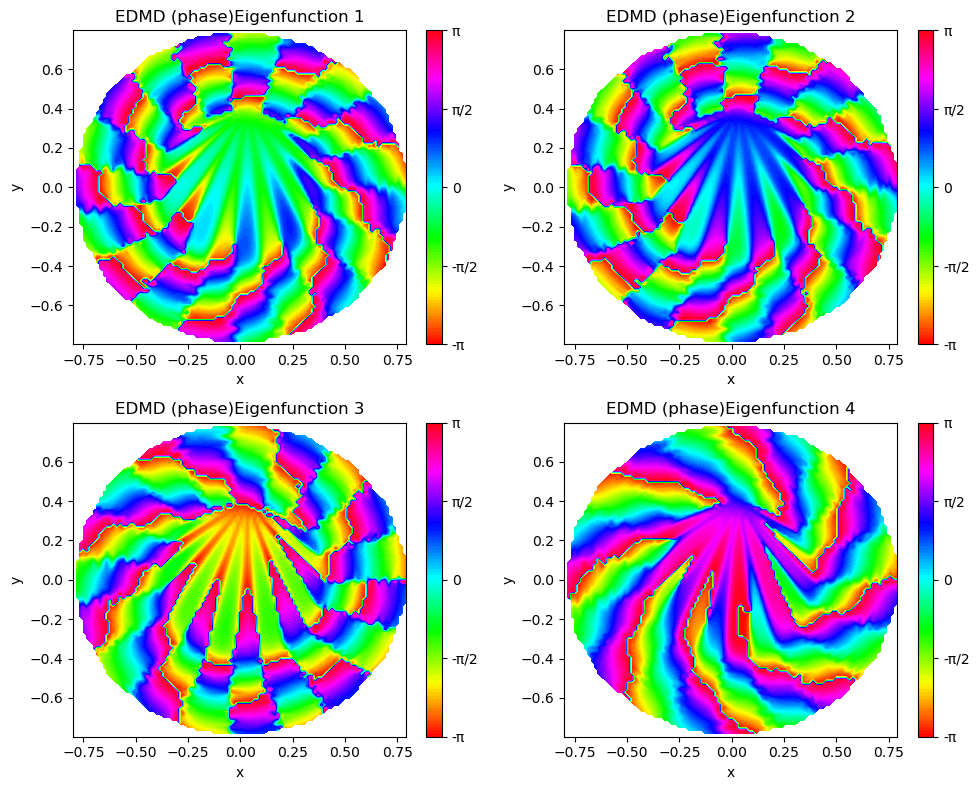}
        \caption{Eigenfunctions by EDMD}
        \label{fig: 2d_sl_eqn_efun_edmd}
    \end{subfigure}
    \hfill
    \begin{subfigure}[b]{0.33\textwidth}
        \includegraphics[width=\textwidth]{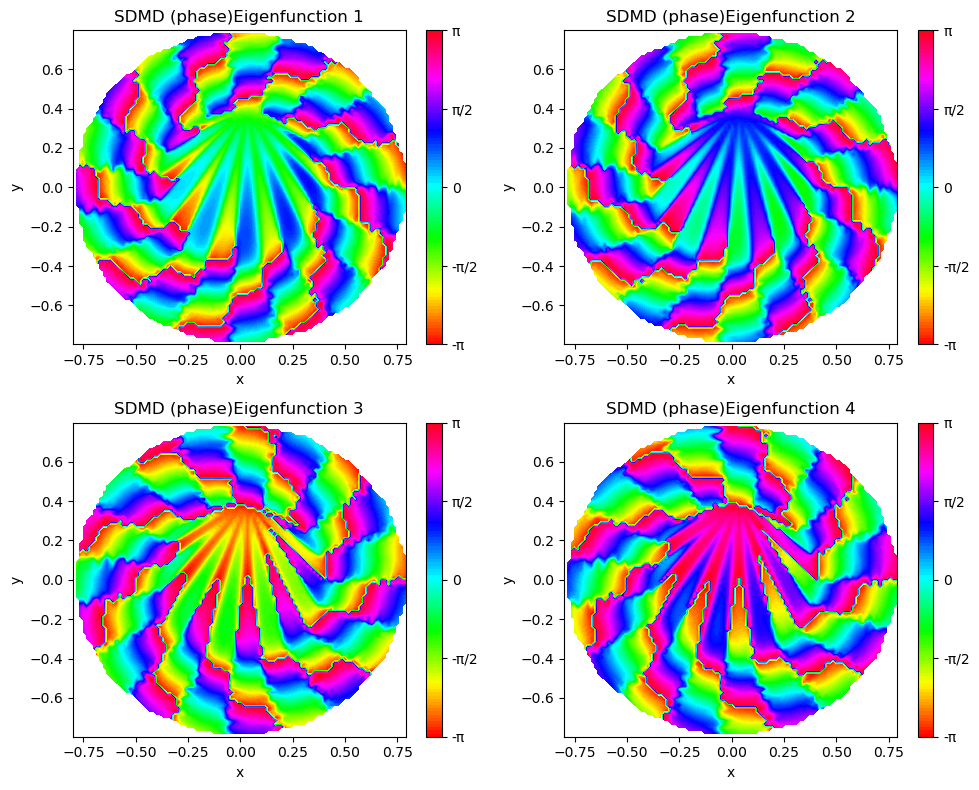}
        \caption{Eigenfunctions by SDMD}
        \label{fig: 2d_sl_eqn_efun_sdmd}
    \end{subfigure}
    \hfill
    \begin{subfigure}[b]{0.33\textwidth}
        \includegraphics[width=\textwidth]{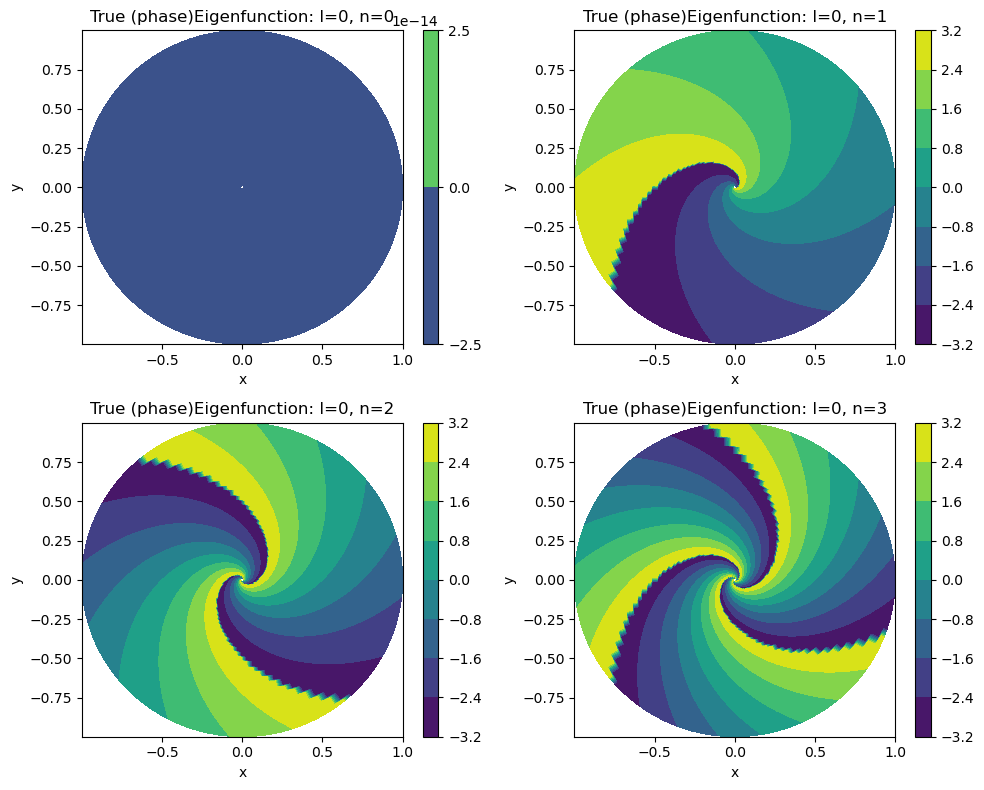}
        \caption{True eigenfunctions}
        \label{fig: 2d_sl_eqn_true}
    \end{subfigure}
    \caption{Comparison of eigenfunctions estimated by EDMD and SDMD for stochastic Stuart-Landau system with Fourier basis. \eqref{fig: 2d_sl_eqn_efun_edmd} shows the eigenfunctions obtained from EDMD, while \eqref{fig: 2d_sl_eqn_efun_sdmd} shows those obtained from SDMD.}
    \label{fig: 2d_sl_eqn_efun_comparison}
\end{figure}

\section{Estimation of Drift and Diffusion Coefficients in SDE}\label{appendix: sde_coef_estimation}

We start with SDE given in Eq.\eqref{eq: sde}, discretized via the Euler-Maruyama (EM) method over a small time step \(\Delta t\):
\[
    \mathbf{X}_{t+\Delta t} - \mathbf{X}_t = \mathbf{b}(\mathbf{X}_t) \Delta t + \sigma(\mathbf{X}_t) \sqrt{\Delta t} \xi_t.
\]

Here, \(\xi_t \sim \mathcal{N}(0, I)\) introduces noise, and this generates our time series data pairs, \(\mathbf{X}_t\) and \(\mathbf{X}_{t+\Delta t}\). A single neural network is trained to predict the next state, taking \(\mathbf{X}_t\) as input and outputting \(\widehat{\mathbf{X}}_{t+\Delta t}\). The loss function driving this training is the mean squared error(SE):
\[
    \text{MSE} = \frac{1}{N} \sum_{t} (\mathbf{X}_{t+\Delta t} - \widehat{\mathbf{X}}_{t+\Delta t})^2.
\]

This measures the average squared difference between the actual next state \(\mathbf{X}_{t+\Delta t}\) from the data and the NN’s prediction \(\widehat{\mathbf{X}}_{t+\Delta t}\), over \(N\) data pairs. By minimizing this loss, the NN learns to approximate the deterministic shift, \(\widehat{\mathbf{X}}_{t+\Delta t} \approx \mathbf{X}_t + \mathbf{b}(\mathbf{X}_t) \Delta t\), as the noise term’s mean is zero.

From this NN, the drift is estimated as
\[
    \mathbf{b}(\mathbf{X}_t) \approx \frac{\widehat{\mathbf{X}}_{t+\Delta t} - \mathbf{X}_t}{\Delta t},
\]
using \(\mathbf{X}_t\) from the data and \(\widehat{\mathbf{X}}_{t+\Delta t}\) from the NN. For diffusion, the residual \(\mathbf{r}_t = \mathbf{X}_{t+\Delta t} - \widehat{\mathbf{X}}_{t+\Delta t}\) captures the noise, with variance defined as \(\textit{variance} = (\mathbf{X}_{t+\Delta t} - \widehat{\mathbf{X}}_{t+\Delta t})^2 \approx \sigma(\mathbf{X}_t)^2 \Delta t\). Thus,
\[
    \sigma(\mathbf{X}_t) \approx \sqrt{\frac{\textit{variance}}{\Delta t}}.
\]

\end{document}